\documentclass[notitlepage,11pt]{article}
	
\usepackage{authblk}
\usepackage[margin=1in]{geometry}
\usepackage{amsmath,amsfonts,amsthm,amssymb}
\usepackage{commath,mathtools,bbm}				
\usepackage{graphicx}				
\usepackage{epstopdf}
\usepackage{subfigure}
\usepackage{url}
\usepackage[mathscr]{euscript}
\usepackage{fontenc}
\usepackage{dsfont}
\usepackage{enumitem}
\hypersetup{linkcolor=red,colorlinks=true,citecolor=blue}
\usepackage{booktabs}
\usepackage[title]{appendix}

\usepackage{color}

\usepackage{caption}
\usepackage{pdflscape}
\usepackage{afterpage}
\usepackage{capt-of}

\usepackage{array}
\usepackage{multirow}
\newcolumntype{C}{>{$\displaystyle} c <{$}}

\usepackage[maxbibnames=99,style=alphabetic,backend=bibtex,sorting=nyt,sortcites,natbib=true]{biblatex}
\usepackage{xargs}
\usepackage{xifthen}

\DeclareMathOperator*{\argmin}{\arg\min}

\DeclareMathOperator{\co}{conv}

\DeclareMathOperator{\proj}{Proj}
\DeclareMathOperator{\relint}{rel.int}
\DeclareMathAlphabet{\mathpzc}{OT1}{pzc}{m}{it}
\DeclareMathOperator{\sign}{sgn}

\newcommand{\concf}[1][\f]{{#1}^{\mathrm{conc}}}
\newcommand{\cvxf}[1][\f]{{#1}_{\B}^{\mathrm{cvx}}}
\newcommand{\concenv}[2]{\operatorname{cav}_{#2}[#1]}
\newcommand{\cvxenv}[2]{\operatorname{vex}_{#2}[#1]}

\newcommand{\x}[1]{#1}
\newcommand{\bx}{\x{x}}

\newcommand{\real}{\mathbb{R}}
\newcommand{\nonnegreal}[1][]{\real^{#1}_{+}}
\newcommand{\posint}{\mathbb{Z}_{\ge 1}}

\newcommand{\power}{\alpha}
\newcommand{\bpower}{\x{\power}}

\renewcommand{\P}[1][]{P}
\renewcommand{\S}{S}
\newcommand{\conv}[1]{\co #1}

\newcommand{\err}[1]{\mu\left(#1\right)}

\newcommand{\degree}[1][]{%
\ifthenelse{\isempty{#1}}%
{d}
{\abs{#1}}
}

\newcommand{\erropt}[1][]{
\ifthenelse{\isempty{#1}}%
{\mathscr{C}^{1}_{\degree}}
{\mathscr{C}(#1)}
}
\newcommand{\erroptcvx}{\mathscr{C}^{2}_{\degree}}

\newcommand{\E}{\mathscr{E}_{r,n}}
\newcommand{\D}{\mathscr{D}_{r,n}}

\newcommand{\zerovec}{\mathbf{0}}
\newcommand{\onevec}[1][]{
\ifthenelse{\isempty{#1}}%
{\mathds{1}}
{\mathbf{e}_{#1}}
}
\newcommand{\multilin}[1][x]{\ensuremath{\prod_{j=1}^{n}{#1}_{j}}} 
\newcommandx{\monom}[2][1=\bpower,2=\bx,usedefault]{{#2}^{#1}}
\newcommandx{\monomexp}[2][1=\power,2=x,usedefault]{\prod_{j=1}^{n}{#2}_{j}^{{#1}_{j}}}

\DeclareMathOperator{\graphing}{\mathcal{G}}
\newcommandx{\graph}[2][1=,2=,usedefault]{
\ifthenelse{\isempty{#2}}%
{\graphing_{#1}(\f[m])}
{\graphing_{#1}(#2)}
}
\newcommandx{\graphmonom}[2][1=\S,2=\f,usedefault]{\graph[#1][#2]}

\newcommand{\f}[1][f]{\mathpzc{#1}}
\newcommand{\g}{\f[g]}
\newcommand{\fmin}[1][\S]{\f^{\min}_{#1}}
\newcommand{\fmax}[1][\S]{\f^{\max}_{#1}}
\newcommand{\linear}[1][\beta]{\ell_{#1}}

\renewcommand{\L}{l}
\newcommand{\U}{u}
\renewcommandx{\H}[2][1=\L,2=\U,usedefault]{[\x{#1},\x{#2}]}

\newcommandx{\Hzone}[3][1=0,2=1,3=n,usedefault]{[#1,#2]^{#3}}
\renewcommand{\c}{c}

\newcommand{\Honer}[1][n]{\Hzone[1][r][#1]} 
\newcommandx{\Hnegoo}[2][1=n,2=1,usedefault]{\Hzone[-#2][#2][#1]} 
\newcommand{\scaledbox}{[\x{\c}\x{\L},\x{\c}\x{\U}]}

\newcommand{\Deltaone}[1][]{
\ifthenelse{\isempty{#1}}%
{\Delta^{\onevec}_{n}}
{\Delta^{\onevec}_{n}(#1)}
}

\newcommandx{\errconv}[2][1=\S,2=\f,usedefault]{\err{\conv{\graph[#1][#2]}}}
\newcommandx{\errvex}[2][1=\S,2=\f,usedefault]{\err{\graph[][\cvxenv{#2}{#1}]}}
\newcommandx{\errcav}[2][1=\S,2=\f,usedefault]{\err{\graph[][\concenv{#2}{#1}]}}

\newcommand{\zopt}[1][\S]{z^{\ast}_{#1}}
\newcommand{\zmon}[1][\S]{z^{mono}_{#1}}

\newtheorem{theorem}{Theorem}[section]
\newtheorem{lemma}{Lemma}[section]
\newtheorem{proposition}{Proposition}[section]
\newtheorem{corollary}{Corollary}[section]
\newtheorem{claim}{Claim}[section]
\newtheorem{observation}{Observation}[section]
\theoremstyle{remark}
\newtheorem{remark}{Remark}

\theoremstyle{definition}

\newcommand{\ouraffil}{
\emph{\hspace{-0.7cm} Department of Mathematical Sciences, Clemson University}\\
\emph{Email address}: \texttt{\{wadams,agupte,yibox\}@clemson.edu}\\
}

\newcommand{\blue}[1]{#1}

\title{Error bounds for monomial convexification in \linebreak polynomial optimization}
\author{Warren Adams}
\author{Akshay Gupte}
\author{Yibo Xu}
\affil{\footnote{\ouraffil}}
\date{November 23, 2017}

\begin{document}

{
\renewcommand{\footnotemark}{}
\maketitle
}

\begin{abstract}
Convex hulls of monomials have been widely studied in the literature, and monomial convexifications are  implemented in global optimization software for relaxing polynomials. However, there has been no study of the error in the global optimum from such approaches. We give bounds on the worst-case error for convexifying a monomial over subsets of $\Hzone$. This implies additive error bounds for relaxing a polynomial optimization problem by convexifying each monomial separately. Our main error bounds  depend primarily on the degree of the monomial, making them easy to compute. Since monomial convexification studies depend on the bounds on the associated variables, in the second part, we conduct an error analysis for a multilinear monomial over two different types of box constraints. As part of this analysis, we also derive the convex hull of a multilinear monomial over $\Hnegoo$.
 \smallskip

\end{abstract}

\section{Introduction}\label{sec:intro}
A polynomial $p\in\real[x]$, where $\real[x]=\real[x_{1},\dots,x_{n}]$ is the ring of $n$-variate polynomials, is a linear combination of monomials and is expressed as  $p(x) = \sum_{\power}c_{\power}\monom$ where the sum is finite, $\monom := \prod_{j=1}^{n}x_{j}^{\power_{j}}$ is a monomial,  and every $\power_{j}$ is a nonnegative  integer. A polynomial optimization problem is 
\[\zopt = \min\,\{p(x)\mid x\in\S\}\] for a compact convex set $\S$ and $p \in \real[x]$. It is common to assume that the degree of the polynomial is bounded by some constant $m$ and this is denoted by $p\in\real[x]_{m}$. Polynomials, in general, are nonconvex functions, thereby necessitating the use of  global optimization algorithms for optimizing them. Strong and efficiently computable convex relaxations are a major component of these algorithms, making them a subject of ongoing research. 
One approach for devising good relaxations is based on taking the convex envelope of each polynomial $p(x)$ over $\S$. However, since this computation is NP-hard even in the most basic cases having $m=2$ and $\S=\Hzone$ or $\S$ being a  standard simplex, 
a main emphasis of the envelope studies has been on finding the envelope either under  structural assumptions on $\S$ or by considering only a subset of all the monomials appearing in $p(x)$. Also, one is interested in obtaining polyhedral relaxations of the envelope so that lower bounds can be computed cheaply by solving linear programs (LPs) iteratively \citep{locatelli2014convex,meyer2005convex,tawarmalani2013explicit,sherali2012poly}. If $p(x)$ is a multilinear polynomial (i.e. $\power_{j}\in\{0,1\}$ for all $j$) and $\S$ is a box, then the envelopes are polyhedral and we know exponential sized extended formulations \citep{rikun1997convex,sherali1997convex}, as well as valid inequalities \citep{del2016polyhedral,crama2017class} and  efficient cutting planes \citep{misener2015dynamically,bao2015global} in projected spaces.  A second method for obtaining lower bounds on the polynomial optimization problem has been to use the moments approach and \citet{lasserre2001global} hierarchy of semidefinite relaxations (SDPs) that converges to the global optimum \citep{lasserre2015introduction,laurent2009sums}. All of these techniques can of course also be used for relaxing a optimization problem that has polynomials in both the objective and constraints.

For a general polynomial $p(x)=\sum_{\power}c_{\power}\monom$, given  that it is hard to find the envelope explicitly and that computability of the SDP bounds does not scale well, a common relaxation technique, motivated by the classical work of \citet{mccormick1976computability}, has been to replace each monomial $\monom$  with a continuous variable, say $w$, and then add inequalities to convexify the graph of $\monom$ over $\S$, which is the set $\{(x,w)\in\S\times\real\mid w = \monom \}$. This is referred to as monomial convexification, and it typically yields a weaker relaxation than the envelope of the polynomial due to the fact that the envelope operator does not distribute over sums in general. However, because they may be cheaper and easier to generate than convexification of the entire polynomial, convex hulls of monomials have received significant attention \citep{bao2015global,buchheim2016monomial,liberti2003convex,couenne} and are also routinely implemented in leading global optimization software \citep{dalkiran2016rlt,misener2014antigone,tawarmalani2005polyhedral}. We still do not know an explicit form for the convex hull of a general monomial, but a number of  results are available for bivariate monomials \citep{locatelli2016polyhedral} and $n$-variate multilinear monomials \citep{belotti2010valid,al1983jointly,benson2004concave,crama1993concave,mahdi2010coloring,meyer2004trilinear,ryoo2001analysis}. Moreover, there also exist challenging applications \citep{buchheim2010integer} where the constraints can be formulated as having only monomial terms, thereby making monomial convexifications necessary for obtaining strong relaxations. 

To quantify the strength of a relaxation of $p(x)$, one is interested in bounding the error produced with respect to the global optimum $\zopt$ by optimizing over this relaxation. Error bounds for converging solutions of  iterative optimization algorithms have been the subject of study  before \citep{pang1997error}, but since these are not suited for studying relaxation strengths, different error measures have been proposed. \citet{mahdi2010coloring} studied a relative error measure for the relaxation of a bilinear polynomial $p\in\real[x]_{2}$ over $\S=\Hzone$ obtained by convexifying each monomial with its McCormick envelopes. They showed that for every $x\in\Hzone$, the ratio of the difference between the McCormick overestimator and underestimator values at $x$ and the difference between the concave and convex envelope values at $x$ can be bounded by a constant that is solely in terms of the chromatic number of the co-occurrence graph of the bilinear polynomial. Recently, \citet{boland2017bounding} showed that this same ratio cannot be bounded by a constant independent of $n$. Another, and somewhat natural, way of measuring the error from a relaxation is to bound the absolute gap $\zopt - \tilde{z}_{\S}$, where $\tilde{z}_{\S}$ is a lower bound on $\zopt$ due to some convex relaxation of $\{(x,w)\in\S\times\real\mid w = p(x)\}$. Such a bound helps determine how close one is to optimality in a global optimization algorithm. Also, there are examples (cf. $\multilin$ over $\Honer$ in \citep[pp. 332]{mahdi2010coloring}) where the relative error gap of McCormick relaxation goes to $\infty$, while this can never happen with the  absolute gap. The only result that we know of on bounding absolute gaps for general polynomials is due to \citet{de2010error} who used  Bernstein approximation of polynomials for a hierarchy of LP and SDP relaxations. (On the contrary, \citep{de2016convergence,de2015error} bound the absolute error from upper bounds on $\zopt$.). We mention that the absolute errors arising from piecewise linear relaxations of bilinear monomials appearing in a specific application were studied by \citet{deygupte2013pooling}. Finally, a third error measure is based on comparing the volume of a convex relaxation to the volume of the convex hull. This has been done for McCormick relaxations of a trilinear monomial over a box by \citet{speakman2015quantifying}. 

\paragraph{Our contribution.} 
In this paper, we bound the absolute gap to $\zopt$ from monomial convexification and thereby add to the small number of explicit error bounds for polynomial optimization. To bound this gap, we analyze the error in relaxing a monomial with its convex hull. This error analysis not only implies a bound on the absolute gap to $\zopt$ but it also can be used for bounding the error in relaxing any optimization problem with polynomials in both the objective and constraints. Our error measure is the maximum absolute deviation between the actual value and the approximate value of the monomial. Thus for any set $X$ in the $(x,w)$-space, we denote the error of $X$ with respect to $\monom$ by $\err{X}$, which is defined as 
\begin{equation}\label{eq:err}
\err{X} := \max_{(x,w)\in X}\;\abs{w - \monom}. 
\end{equation}
We will mostly be interested in the error $\err{\cdot}$ for the convex hull of the graph of $\monom$ and for the convex and concave envelopes of $\monom$.  As mentioned earlier, monomial convexification errors have gone largely unnoticed in the literature, the only results being for the bilinear monomial $x_{1}x_{2}$. The folklore result \citep[cf.][]{al1983jointly} for  $x_{1}x_{2}$ over a rectangle $[\L_{1},\U_{1}]\times[\L_{2},\U_{2}]$ states that the convex hull and envelope errors are attained at $(x_{1},x_{2}) = (\frac{\U_{1} + \L_{1}}{2}, \frac{\U_{2}+\L_{2}}{2})$, which is the midpoint of the two diagonals of the box.  \citet{linderoth2005simplicial} derived error formulae for $x_{1}x_{2}$ over triangles created by the two diagonals of $[\L_{1},\U_{1}]\times[\L_{2},\U_{2}]$. Since convex hull and envelope results for a  bilinear polynomial are invariant to affine transformations, it is equivalent to consider $x_{1}x_{2}$ over $\Hzone[][][2]$. 
Substituting $n=2$ and $\power_{1}=\power_{2}=1$ in our forthcoming error bounds recover these    known errors.


\paragraph{Notation.}
The vector of ones is $\onevec$, the $i^{th}$ unit coordinate vector is $\onevec[i]$, and the vector of zeros is $\zerovec$; the dimensions will be apparent from the context in which these vectors are used. 
The convex hull of a set $X$ is $\conv{X}$ and the relative interior of $\conv{X}$ is $\relint X$.  A nonempty box in $\real^{n}$ is $\H := [\L_{1},\U_{1} ] \times \dots \times [\L_{n},\U_{n} ]$. The standard boxes that we focus on in this paper are $\Hzone,[-1,1]^{n}$, and $[1,r]^{n}$, for arbitrary scalar $r>1$. Another compact convex set of interest to us is the standard $n$-simplex $\Delta_{n} := \conv{\{\zerovec,\onevec[1],\dots,\onevec[n] \}} = \left\{x\ge\zerovec\mid \sum_{j=1}^{n}x_{j} \le 1 \right\}$. For convenience, we write $\f(x) := \monom$, $\fmin := \min_{x\in\S}\f(x)$, $ \fmax := \max_{x\in\S}\f(x)$. The convex envelope of $\monom$ over $\S$, which is defined as the pointwise supremum of all convex underestimators of $\monom$ over $\S$, is denoted by $\cvxenv{\f}{\S}$. The concave envelope, which is analogously defined, is $\concenv{\f}{\S}$. The graph of a function $g(x)$ with domain $\S$ is denoted by $\graph[\S][g] := \{(x,w)\in\S\times\real\mid w = g(x) \}$. The graphs of the monomial and its envelopes are $\graphmonom$, 
$\graph[][\cvxenv{\f}{\S}]$ 
and $\graph[][\concenv{\f}{\S}]$. 
Two special types of monomials are the symmetric monomial and the multilinear monomial. The former has $\power = \power_{0}\onevec$ for some $\power_{0}\in\posint$, and the latter, denoted by $\f[m](x) := \multilin$, is a special case of the former with $\power=\onevec$. 
For $\beta\in\nonnegreal[n]$, we denote $\degree[\beta] := \sum_{j=1}^{n}\beta_{j}$.

\subsection{Main results}
We obtain strong and explicit upper bounds on $\err{\cdot}$ for different types of monomials. In the polynomial optimization literature, it is common to assume, upto scaling and translation, that the domain $\S$ of the problem is a subset of $\Hzone$. When analyzing a single monomial, this assumption is not  without loss of generality since the monomial basis of $\real[x]$ is not closed upto translating and scaling the variables. 
Hence we divide our analysis into two parts. First, we consider a general monomial $\f(x)=\monom$ over a compact convex set $\S\subseteq\Hzone$, and bound the errors without using explicit analytic forms of the envelopes, which  are hard to compute and unknown in closed form for arbitrary $\S$. The concave error is bounded by computing the error from a specific concave overestimator that is precisely the concave envelope of $\monom$ over $\Hzone$.  On the convex side, we bound the error for any convex underestimator given as the pointwise supremum of (possibly uncountably many) linear functions, each of which underestimates $\monom$ over $\S$. Thus our error analysis has a  distinctly polyhedral flavor. 

In the second part, we limit our attention to a multilinear monomial $\f[m](x) = \multilin$, but the domain $\S$ is either a box with constant ratio or a symmetric box. By a box with constant ratio, we mean any box $\H$ for which there exists a scalar $r > 1$ such that $\U_{i}/\L_{i} = r$ for all $i$ with $\L_{i}>0$, and $\L_{i}/\U_{i} = r$ for all $i$ with $\L_{i}<0$. By a symmetric box, we mean any box $\H$ that has $\U_{i}=-\L_{i}$ for all $i$. Since these boxes are simple scalings of $\Honer$ and $\Hnegoo$, respectively, and our error measure $\err{\cdot}$ scales, we restrict our attention to only $\Honer$ and $\Hnegoo$. 
Contrary to the first part, 
here we first derive explicit polyhedral characterizations of the envelopes and convex hulls over $\Honer$ and $\Hnegoo$ 
and use them to perform a tight error analysis. The polyhedral representations for the $\Honer$ case follow from the literature, whereas those over $\Hnegoo$ 
 are established in this paper.

\subsubsection{General monomial}\label{sec:mainmonom}
Consider a monomial $\monom$ with $\power_{j}\in\posint$ for all $j$. The degree of this monomial is $\degree := \degree[\power] = \sum_{j=1}^{n}\power_{j}$. The following constants will be useful throughout the paper:
\begin{equation}\label{eq:C}
\erropt := \left(1 - \frac{1}{\degree} \right)\,\degree^{\frac{1}{1-\degree}}, 
\qquad \erroptcvx := \left(1 - \frac{1}{\degree} \right)^{\degree}.
\end{equation}

\begin{theorem}\label{thm:converr01}
For the monomial $\f(x)=\monom$ over $\S\subseteq\Hzone$, we have 
\[
\errvex \le \left(1 - \frac{1}{\degree[\gamma]}\right)^{\degree[\gamma]} \le \erroptcvx , \quad \errcav \le \errconv \le \erropt,
\] where for $\sigma_{j} := 1 - \max\{x_{j}\mid x \in \S\}$, we define \[
\gamma_{j}:=
\begin{cases}
\displaystyle\frac{1 - (1-\sigma_{j})^{\power_{j}}}{\sigma_{j}}, & \text{if $\sigma_{j}>0$},\\
\power_{j}, & \text{if $\sigma_{j}=0$,}
\end{cases} \quad j=1,\dots,n.
\] 

If $\zerovec,\onevec\in\S$, then $\errconv = \errcav = \erropt$.
\end{theorem}

The monotonicity of $\erropt$ and $\erroptcvx$ with respect to $\degree$ suggests the intuitive result that convexifying higher degree monomials will likely produce greater errors. As $\degree \to\infty$, we have $\erropt\to 1$ and $\erroptcvx\to 1/e$.

The bounds $\erropt$ and $\erroptcvx$ depend only on the degree of the monomial. They are a consequence of some general error bounds, established in Theorem~\ref{thm:concub} for the concave error and in Theorem~\ref{thm:errenv01} for the convex error, that depend on how the monomial behaves over the domain $\S$. The arguments used in proving Theorem~\ref{thm:converr01} also imply that a family of convex relaxations of $\graphmonom$ has error equal to $\erropt$. We show this in Proposition~\ref{prop:Perr}. We also guarantee in Corollary~\ref{corr:converr01multi} that the convex envelope error bound $\erroptcvx$ is tight for $\f[m](x)$ over $\S=\Hzone$. 

Theorem~\ref{thm:converr01} has two immediate implications. First, we obtain the error in convexifying a monomial over $\Hzone$.
\begin{corollary}
$\errconv[\Hzone] =\erropt$.
\end{corollary}

Second, we obtain an additive error bound on polynomial optimization over subsets of $\Hzone$. For a polynomial $p=\sum_{\power}c_{\power}\monom\in\real[x]$, denote
\begin{equation}\label{eq:L}
L^{\prime}(p) = \max\left\{\max_{\power\colon c_{\power}>0} c_{\power}\erroptcvx, \, \max_{\power\colon c_{\power}<0} -c_{\power}\erropt \right\}.
\end{equation}
Let $\zmon := \min\{\sum_{\power}c_{\power}w_{\power}\mid (x,w_{\power})\in\conv{\graphmonom} \ \, \forall \power \}$ be the lower bound\footnote{To avoid tediousness and with a slight abuse of notation, for each monomial  we write $(x,w_{\power})\in\conv{\graphmonom}$ with the understanding that those $x_{j}$ that appear in the monomial are included.} from monomial convexification on the global optimum $\zopt= \min_{x\in\S}p(x)$.
 
\begin{corollary}\label{corr:L}
For any $p\in\real[x]_{m}$ and compact convex $\S\subseteq\Hzone$, \[\zopt - \zmon \le L^{\prime}(p)\binom{n+m}{n}. \]
\end{corollary}
\begin{proof}
We have $\zmon = \sum_{\power\colon c_{\power}>0}c_{\power}\cvxenv{\f}{\S}(x) + \sum_{\power\colon c_{\power}<0}c_{\power}\concenv{\f}{\S}(x)$. Therefore, \[\zopt-\zmon = \sum_{\power\colon c_{\power}>0}c_{\power}(\monom - \cvxenv{\f}{\S}(x)) + \sum_{\power\colon c_{\power}<0}(-c_{\power})(\concenv{\f}{\S}(x) -\monom).\] Applying Theorem~\ref{thm:converr01} and the construction of $L^{\prime}(p)$ gives us $\zopt-\zmon \le L^{\prime}(p)\sum_{\power}1$. Since $p\in\real[x]_{m}$, there are at most $\binom{n+m}{n}$ monomials in $p(x)$, leading to the claimed error bound. 
\end{proof}

Computing $L^{\prime}(p)$ may get tedious if $p(x)$ has a large number of monomials. A cheaper bound is possible by considering only the largest coefficient in $p(x)$.
\begin{corollary}\label{corr:L2}
For any $p\in\real[x]_{m}$ and compact convex $\S\subseteq\Hzone$, \[\zopt - \zmon \le \max_{\power}\,\abs{c_{\power}}\,\left(1 - \frac{1}{m} \right)m^{\frac{1}{1-m}}\binom{n+m}{n}.\]
\end{corollary}
\begin{proof}
Follows from Corollary~\ref{corr:L} after using $\degree\le m$ and $\erropt$ being monotone in $\degree$.
\end{proof}

The bounds from Theorem~\ref{thm:converr01}, although applicable to arbitrary $\S\subseteq\Hzone$, can be weak if $\zerovec\in\S$ and $\onevec\notin\S$. To emphasize this, we consider a monomial over the standard simplex $\Delta_{n}$ and obtain error bounds that depend on not just the degree of the monomial but also the exponent of each variable. These bounds are stronger than the bounds $\erropt$ and $\erroptcvx$.

\begin{theorem}\label{thm:simplex}
\[ 
\err{\graph[][\concenv{\f}{\Delta_{n}}]} \le \err{\conv{\graph[\Delta_{n}][\f]}} \le \frac{(\monom[][\power])^{1/\degree}}{\degree} \, -\, \frac{\monom[][\power]}{\degree^{\degree}}, \quad \err{\graph[][\cvxenv{\f}{\Delta_{n}}]} = \frac{\monom[][\power]}{\degree^{\degree}}. 
\]
All of the above bounds are tight for a symmetric monomial. 
\end{theorem}

\subsubsection{Multilinear monomial}\label{sec:mainmulti}
Consider the multilinear monomial $\f[m](x) = \multilin$.

\begin{theorem}\label{thm:converr1r}
Denote 
\begin{equation*}
\begin{split}
\D := \max_{i=1,\dots, n-1}\, \left\{\left(1+\frac{i}{n}(r-1)\right)^n \,-\, r^{i}\right\},\quad
\E := 1\,+\,\frac{r^n-1}{(r-1)}\left[\frac{n-1}{n}{\left(\frac{r^n-1}{n(r-1)}\right)}^{\frac{1}{n-1}}\,-\,1\right].
\end{split}
\end{equation*}
For $\f[m](x)$ over $\Honer$,  \[ \errcav[\Honer][{\f[m]}] = \E, \quad \errvex[\Honer][{\f[m]}] = \D, \quad \errconv[\Honer][{\f[m]}] = \max\{\D,\E\}.\] All bounds are attained only on $\relint{\{\onevec,r\onevec \}}$.
\end{theorem}

We conjecture that $\D\le\E$ for all $r,n$ and provide a strong empirical evidence in support of this claim. We prove this conjecture to be asymptotically true by showing that $\lim_{n\to\infty}\D/\E \le 1/e$.

For $\S=\Hnegoo$, we characterize the convex hull in Theorem~\ref{thm:convminusplus} and show that it has the following errors.
\begin{theorem}\label{thm:converr11}
For $\f[m](x)$ over $\Hnegoo$, \[\errcav[\Hnegoo][{\f[m]}] = \errvex[\Hnegoo][{\f[m]}] = \errconv[\Hnegoo][{\f[m]}] = 1+\left(\frac{n-2}{n}\right)^{n}.\] 
This maximum error is attained at all the $2^{n}$ reflections of the point $(\frac{n-2}{n}\onevec,-1)$. 
\end{theorem}
The exact description of the reflected points will be provided when we prove this theorem. Taking $n\to\infty$, this error approaches $1 + 1/e^{2}$ from below.

\subsubsection{Outline}
Our analysis begins with some preliminaries on the error measure. We observe that the error scales with the box and present a lower bound on the error, which we remark is also the proposed upper bound for the two cases $\S=\Hzone$ and $\S=\Honer$. We also formally note the intuition that the convex hull error can be computed as the maximum of the two envelope errors, due to which our error analysis in the remainder of the paper involves analyzing the concave envelope and  the convex envelope separately. \textsection\ref{sec:concerr} and \textsection\ref{sec:cvxerr} analyze these errors for a general monomial $\monom$ over $\S\subseteq\Hzone$. The main error bounds presented in \textsection\ref{sec:mainmonom} are proved in \textsection\ref{sec:converr} and we compare them to those from literature in \textsection\ref{sec:compare}. The multilinear monomial over $\Honer$ and $\Hnegoo$ is analyzed in \textsection\ref{sec:err1r} and \textsection\ref{sec:err11}.

\section{Preliminaries on $\err{\cdot}$}
The error defined in \eqref{eq:err} is obviously monotone with respect to set inclusion: $\err{X_{1}} \le \err{X_{2}}$ for any $X_{1}\subseteq X_{2}$. This enables us to upper bound the convex hull error by using $\err{\conv{\graphmonom}} \le \err{X}$ for any convex relaxation $X$ of $\graphmonom$, and also implies that the convex hull error over a smaller variable domain is upper bounded by the convex hull error over a larger domain. Another property we observe is that computing the convex hull error is equivalent to computing the error due to the convex envelope $\cvxenv{\f}{\S}$ and that due to the concave envelope $\concenv{\f}{\S}$. This intuitively seems correct given the  well-known fact that $\conv{\graphmonom} = \{(\bx,w)\in\S\times\real\mid \cvxenv{\f}{\S}(x) \le w \le \concenv{\f}{\S}(x)\}$, and the fact that the monomial convexification and envelope errors are
\begin{equation*}\begin{split}
&\errconv =\max_{(x,w)\in\conv{\graphmonom}}\,\abs{w - \monom} , \quad \errvex = \max_{x\in\S}\,\monom - \cvxenv{\f}{\S}(x), \\
& \errcav = \max_{x\in\S}\,\concenv{\f}{\S}(x) - \monom.
\end{split}\end{equation*}

\begin{observation}\label{obs:errmaxenv}
Let $X := \{(x,w)\in\S\times\real\mid f_{1}(x) \le w \le f_{2}(x) \}$, where $f_{1}$ and $f_{2}$ are, respectively, convex and concave continuous functions with $f_{1}(x) \le \monom \le f_{2}(x)$ for all $x\in\S$. Then \[
\err{\conv{\graphmonom}} \le \err{X} = \max\left\{\err{\graph[\S][f_{1}]} , \,\err{\graph[\S][f_{2}]} \right\}, 
\] and equality holds if $f_{1} = \cvxenv{\f}{\S}$ and $f_{2}=\concenv{\f}{\S}$.
\end{observation}
The proof is straightforward and is left to the reader. Based on this observation, our error analysis in the rest of the paper involves analyzing the concave envelope and  the convex envelope separately.

A third and final  property we note is that the error scales with the box. For $\x{\c}\in\real^{n}_{\neq 0}$, \[\scaledbox := \{\bx\in\real^{n}\mid \c_{j}\L_{j}\le x_{j}\le \c_{j}\U_{j} \ \forall j \text{ s.t. } \c_{j}>0, \; \c_{j}\L_{j}\ge x_{j}\ge \c_{j}\U_{j} \ \forall j \text{ s.t. } \c_{j}<0 \}\] is the coordinate-wise  scaled version of $\H$. The bijective linear map $\mathscr{D}(\bx) := (\c_{1}x_{1},\ldots,\c_{n}x_{n})$ gives us the relation $\scaledbox=\mathscr{D}(\H)$. Denote $\monom[][\c] := \monomexp[][\c]$.

\begin{observation}\label{obs:errscale}
For any $\c\in\real^{n}_{\neq 0}$, we have $\err{\conv{\graphmonom[{\scaledbox}]}} = \abs{\monom[][\c]}\err{\conv{\graphmonom[\H]}}$, with $(x,w)$ being optimal to $\err{\conv{\graphmonom[\H]}}$ if and only if $(\mathscr{D}(x),\monom[][\c]w)$ is optimal to $\err{\conv{\graphmonom[{\scaledbox}]}}$.
%
\end{observation}

Observation~\ref{obs:errscale} allows us to focus on boxes with specific bounds $\L_{j}$ and $\U_{j}$, and to then extend to slightly more general boxes via scalings.  In particular, error results for 
\begin{itemize}
\item $\Hzone$ scale to any box having a vertex at $\zerovec$,
\item $\Honer$ scale to any box for which the ratio between lower and upper bounds is the same positive scalar in each coordinate, and 
\item $\Hnegoo$ scale to any box that is symmetric with respect to $\zerovec$.
\end{itemize}

Finally, we observe a lower bound on $\errcav$, and hence on $\err{\conv{\graphmonom}}$, when $\S$  contains two points on the ray $\{t\onevec\mid t \ge 0 \}$, which happens for example when $\S = \Hzone[t_{1}][t_{2}]$ for some $t_{1} < t_{2}$ with $t_{2} > 0$.

\begin{lemma}\label{lem:lowerbd}
Suppose $\S\cap \{t\onevec\mid t \ge 0 \} \neq\emptyset$ and let $t_{1},t_{2}\ge 0$ be the minimum and maximum values such that $t_{1}\onevec,t_{2}\onevec\in\S$. Let $\tilde{f}$ be a concave overestimator of $\monom$ on $\S$ and let $X$ be a convex relaxation of $\graphmonom$. Then, $\err{\graph[\S][\tilde{f}]} \ge \phi(\xi^{\prime})$ and $\err{X} \ge \phi(\xi^{\prime})$, where $\phi\colon \xi\in[0,1] \mapsto t_{1}^{\degree} + (t_{2}^{\degree}-t_{1}^{\degree})\xi - (t_{1} + (t_{2}-t_{1})\xi)^{\degree}$ and 
\[\xi^{\prime} = \left(\frac{t_{2}^{\degree}-t_{1}^{\degree}}{\degree}\right)^{\frac{1}{\degree-1}}(t_{2}-t_{1})^{\frac{\degree}{1-\degree}} \;-\; \frac{t_{1}}{t_{2}-t_{1}}.
\] 
\end{lemma}
\begin{proof}
The assumption $t_{1}\onevec,t_{2}\onevec\in\S$ implies $(t_{1}\onevec,t_{1}^{\degree}) , (t_{2}\onevec,t_{2}^{\degree})\in\graphmonom$. Convexity of $X$ and  $\graphmonom\subset X$ lead to $(((1-\xi)t_{1} + \xi t_{2})\onevec, (1-\xi)t_{1}^{\degree} + \xi t_{2}^{\degree}) \in X$ for all $\xi\in[0,1]$. Therefore
\[
\begin{split}
\err{X} \;\ge\; \max_{0\le\xi\le 1}\,\abs{(1-\xi)t_{1}^{\degree} + \xi t_{2}^{\degree} - ((1-\xi)t_{1} + \xi t_{2})^{\degree}} &\;=\; \max_{0\le\xi\le 1}\,  (1-\xi)t_{1}^{\degree} + \xi t_{2}^{\degree} - ((1-\xi)t_{1} + \xi t_{2})^{\degree}\\
&\;=\; \max_{0\le\xi\le 1}\, \phi(\xi),
\end{split}
\] 
where the equality is due to $t_{1},t_{2}\ge 0$ and convexity of the function $t\mapsto t^{\degree}$ on $\nonnegreal$. Since $\phi(0)=\phi(1)=0$, by Rolle's theorem, there exists a stationary point in $[0,1]$ and this point is exactly $\xi^{\prime}$ stated above. Since $\phi$ is concave, $\xi^{\prime}$ must be a maxima.
For a concave overestimator $\tilde{f}$, we have  \[
\begin{split}
\err{\graph[\S][\tilde{f}]} = \max_{x\in\S}\,\tilde{f}(x)-\monom &\;\ge\; \max_{0\le\xi\le 1}\tilde{f}((1-\xi)t_{1} + \xi t_{2}) - ((1-\xi)t_{1} + \xi t_{2})^{\degree}\\
 &\;\ge\; \max_{0\le\xi\le 1}(1-\xi)\tilde{f}(t_{1}) + \xi \tilde{f}(t_{2}) - ((1-\xi)t_{1} + \xi t_{2})^{\degree} \\
&\;\ge\; \max_{0\le\xi\le 1}(1-\xi)t_{1}^{\degree} + \xi t_{2}^{\degree} - ((1-\xi)t_{1} + \xi t_{2})^{\degree} \\
&\;=\;\phi(\xi^{\prime}).\qedhere
\end{split}
\] 
\end{proof}
 
\begin{remark}\label{rem:lowerbd}
For lower bounding $\err{X}$, the above proof really only requires $(t_{1}\onevec,t_{1}^{\degree}) , (t_{2}\onevec,t_{2}^{\degree})\in X$. The stronger assumption $t_{1}\onevec,t_{2}\onevec\in\S$ is made for convenience.
\end{remark}

\begin{remark}
The above method of lower bounding the error can also be utilized by considering arbitrary $l,u\in\S$ with $\zerovec \le l\le u$. This generalization is made possible by the observation that the function $\xi\mapsto \prod_{j=1}^{n}(l_{j}+(u_{j}-l_{j})\xi)^{\power_{j}}$ is convex over $\nonnegreal$. Since the derivation gets extremely tedious and does not yield new insight, we omit the general case here.
\end{remark}

Substituting $t_{1}=0,t_{2}=1$ in Lemma~\ref{lem:lowerbd} yields the critical point to be $\xi^{\prime}=(1/\degree)^{1/(\degree-1)}$ so that $\phi(\xi^{\prime}) = \xi^{\prime} - {\xi^{\prime}}^{\degree} = \xi^{\prime}(1 - {\xi^{\prime}}^{\degree-1}) = (1/\degree)^{\frac{1}{\degree-1}}( 1 - 1/\degree ) = \erropt$, where the constant $\erropt$ was introduced in equation~\eqref{eq:C}. Thus the significance of the lower bound  from this lemma is that we prove in Theorem~\ref{thm:converr01} that it is indeed equal to the maximum error of the convex hull when $\{\zerovec,\onevec\}\subset\S\subseteq\Hzone$. For a multilinear monomial over $\S=\Honer$ for some $r>1$, or equivalently $\S=\Hzone[\frac{1}{r}][1]$ using the scaling from Observation~\ref{obs:errscale}, the constant $\E$ defined in the statement of Theorem~\ref{thm:converr1r} is exactly the lower bound obtained from Lemma~\ref{lem:lowerbd} by substituting $t_{1}=1,t_{2}=r$ and we prove that this is the maximum concave envelope error and conjecture, with strong empirical evidence in support, that it is also the maximum convex hull error.

\newcommand{\B}{\mathcal{B}}
\newcommand{\K}[1][\B]{\mathscr{K}(#1)}
\newcommand{\bb}{\kappa}
\newcommand{\bbmax}{\bb^{\ast}}
\newcommandx{\ratio}[2][1=\beta,2=\bb,usedefault]{r(#1,#2)}

\section{Monomial over $\Hzone$} \label{sec:err01}
This section considers a general multivariate monomial $\monom$, for some $\power\in\posint^{n}$, over a nonempty compact convex set $\S\subseteq\Hzone$. It follows that $\graphmonom\subseteq[0,1]^{n+1}$. Our main error bounds on $\err{\conv{\graphmonom}}$ depend only on the degree $\degree:=\sum_{j=1}^{n}\power_{j}$ of the monomial and therefore are independent of how the monomial behaves on its domain $\S$. However, en route to deriving these formulas, we establish tighter bounds that depend on the minimum and maximum value of $\monom$ over $\S$ and thus are expensive to compute in general. The error formulas for the multilinear case will follow after substituting $\power=\onevec$. Motivated by Observation~\ref{obs:errmaxenv}, we bound the convex hull error by bounding the envelope errors separately. 

Before we begin, we recall that the envelopes of $\f[m](x)$ were shown by \citet{crama1993concave} to be 
\begin{subequations}
\begin{equation}\label{eq:crama}
\cvxenv{\f[m]}{\Hzone}(x) = \max\left\{0, 1 + \sum_{j=1}^{n}(x_{j}-1) \right\}, \quad \concenv{\f[m]}{\Hzone}(x) = \min_{j=1,\dots,n}x_{j}.
\end{equation}
\begin{remark}
The envelopes of $\f[m](x)$ over a box $\H$ having one of its vertices at the origin, i.e., $\L_{j}\U_{j} = 0$ for all $j$, can be obtained by scaling the variables in \eqref{eq:crama} as $x_{j}\leftarrow\U_{j}x_{j}$ for $j\in J_{1}:= \{j\mid \L_{j}=0\}$, $x_{j}\leftarrow\L_{j}x_{j}$ for $j\in J_{2}:= \{j\mid \U_{j}=0\}$, and $w_{j}\leftarrow w \prod_{j\in J_{1}}\U_{j}\prod_{j\in J_{2}}\L_{j}$.
\end{remark}
The concave envelope in \eqref{eq:crama} is also the concave envelope of $\monom$ over $\Hzone$ for every $\power\ge\onevec$, i.e.
\begin{equation}\label{eq:crama2}
\concenv{\f}{\Hzone}(x) = \min_{j=1,\dots,n}x_{j}.
\end{equation}
\end{subequations}
This is  because $\f(x)=\f[m](x)$ for $x\in\{0,1\}^{n}$ and a monomial $\monom$ with $\power\ge\onevec$ is known to be concave-extendable from the vertices of $\Hzone$ (meaning that $\concenv{\f}{\Hzone}$ can be obtained by looking at the values of $\f(x)$ solely at $\{0,1\}^{n}$); see \citep{tawarmalani2002convex}. One can also establish this fact independently without using concave-extendability of $\f(x)$.


For notational convenience throughout this section, we denote \[ E_{0}:= \{\bx\in\Hzone\mid x_{i}=0 \text{ for some } i\}, \quad E_{j} := \conv{\{\onevec,\onevec-\onevec[j]\}}, \ \, j=1,\ldots,n.\] That is, $E_{0}$ is the union of all the coordinate plane facets of $\Hzone$ and  $E_{j}$ is the $j^{th}$ edge of $\Hzone$ that is incident to the vertex $\onevec$.

\subsection{Concave overestimator error}\label{sec:concerr}
Throughout, we consider the piecewise linear concave  function $\concf(x) := \min_{j=1\dots n}\,x_{j}$, which we noted in \eqref{eq:crama2} to be the concave envelope of $\f(x)$ over $\Hzone$. \blue{First, we treat the general case where $\S$ is any subset of $\Hzone$, and later we consider the case of $\S$ being a standard simplex.

\subsubsection{General case}}
For arbitrary $\S\subseteq\Hzone$, we have $\concenv{\f}{\S}(\cdot) \le \concf(\cdot)$ due to $\power\ge\onevec$ and $x\in\Hzone$ implying $0 \le \monom \le  x_{j}^{\power_{j}}\le x_{j}$ for all $j$. We observe that this overestimator is exact only on $E_{0}$ or on edges $E_{i}$'s along which the monomial is linear.

\begin{proposition}\label{prop:lowerbdconc}
$\concf(x) =\monom$ if and only if $x=\onevec$ or $x\in E_{0}$ or $x\in E_{i}$ for some $i$ with $\power_{i}=1$.
\end{proposition}
\begin{proof}
For $\S\subseteq\Hzone$ and $\power\ge\onevec$, $\concf(x)\ge \monom$ follows from the facts  and $x_{i}^{\power_{i}}x_{j}^{\power_{j}}\le x_{i}^{\power_{i}} \ \forall i\neq j$. The equalities $\concf(\onevec)=\monom[][\onevec]=1$ and $\concf(x) = \monom = 0$, for all $x\in E_{0}$, are obvious. For any $x\in\relint E_{i}$, $x_{i} \in(0,1)$ and $x_{j}=1 \ \forall j\neq i$ give us $\concf(x) = x_{i}$ and  $f(x)=x_{i}^{\power_{i}}$. Thus it is obvious that for $x\in\relint E_{i}$, $\concf(x) = f(x)$ if and only if $\power_{i}=1$. Now let $x$ be any point in $\S$ that does not belong to a coordinate plane nor to any edge $E_{i}$. Then there exist distinct indices $i,j$ with $x_{i},x_{j}\in(0,1)$ and $ x_{i} \le x_{j} \le x_{k} \ \forall k\neq i,j$. Therefore $0 < \monom \le x_{i}^{\power_{i}}x_{j}^{\power_{j}} < x_{i} = \concf(x)$.
\end{proof}

Since $\concenv{\f}{\S}(\cdot)\le\concf(\cdot)$, the error due to $\concf$, which is the maximum value of the difference $\concf(x) - \monom$ over $\S$, provides an upper bound on the error from $\concenv{\f}{\S}$.  Proposition~\ref{prop:lowerbdconc} tells us that this maximum difference occurs either in the interior of $\Hzone$ or in the relative interior of some face of $\Hzone$ passing through $\onevec$. In the following result, we give a tight  upper bound on  $\concf(x) - \monom$  that is attained at a specific point on the diagonal between $\zerovec$ and $\onevec$. This is our main error bound for $\concenv{\f}{\S}$.

\begin{theorem}\label{thm:concub}
$\err{\graph[\S][\concf]} \le {\xi^{\prime}}^{1/\degree}-\xi^{\prime}$, where $\xi^{\prime}=\min\{\max\{\fmin,\degree^{\frac{\degree}{1-\degree}} \}, \fmax \}$. This bound can be attained only at the point ${\xi^{\prime}}^{1/\degree}\onevec\in\relint{\{\zerovec,\onevec\}}$ and hence is tight if and only if ${\xi^{\prime}}^{1/\degree}\onevec\in\S$. 
\end{theorem}
\begin{proof}
Since $x\in\Hzone$ and $\power\ge\onevec$, we have $\left(\min_{i} x_{i} \right)^{\degree} \le \monom \le \min_{i}x_{i}$. This implies $\concf(x) \le (\monom)^{\frac{1}{\degree}}$ for $x\in\Hzone$, which leads to 
\begin{equation}\label{eq:concub}
\max_{x\in\S}\,\concf(x) - \monom \;\le\; \max_{x\in\S}\,(\monom)^{\frac{1}{\degree}} - \monom.
\end{equation}
Since $f(x)$ is a continuous function with minimum and maximum values $\fmin$ and $\fmax$ on the closed convex set $\S$, the intermediate value theorem implies that \[\max_{x\in\S}\,(\monom)^{\frac{1}{\degree}} - \monom  = \max\, \{\xi^{\frac{1}{\degree}} - \xi \mid \fmin\le\xi\le \fmax\}.\]
We have $0 \le \fmin \le \fmax\le 1$ due to $\S\subseteq\Hzone$. Elementary calculus tells us that the function $\xi^{1/\degree} - \xi$ is concave on $[0,1]$ with a unique stationary point at $\xi_{0}=\degree^{\frac{\degree}{1-\degree}}$ and is increasing on $[0,\xi_{0})$ and decreasing on $(\xi_{0},1]$. Hence the  maximum value of this function on $[\fmin,\fmax]$ is ${\xi^{\prime}}^{1/\degree}-\xi^{\prime}$, where $\xi^{\prime}=\min\{\max\{\fmin,\degree^{\frac{\degree}{1-\degree}} \}, \fmax \}$. Combining this with \eqref{eq:concub} gives us the desired upper bound.

Now we claim that this bound can be tight only on $\relint{\{\zerovec,\onevec\}}$. Suppose this is not true and there exists a $y\in\S\setminus\relint{\{\zerovec,\onevec\}}$ such that $\concf(y)-\monom[][y] = {\xi^{\prime}}^{1/\degree}-\xi^{\prime}$. The fact that $\xi^{\prime}>0$ and $\degree\ge 2$ makes it obvious that $y\neq\zerovec,\onevec$. Thus $y\notin\conv{\{\zerovec,\onevec\}}$. Since $\concf(y) \le (\monom[][y])^{1/\degree}$ and ${\xi^{\prime}}^{1/\degree}-\xi^{\prime}$ is the maximum value of the right hand side in \eqref{eq:concub}, we have \[
{\xi^{\prime}}^{1/\degree}-\xi^{\prime} = \concf(y)-\monom[][y] \le (\monom[][y])^{1/\degree} - \monom[][y] \le {\xi^{\prime}}^{1/\degree}-\xi^{\prime},
\] implying that equality holds throughout. Hence $\concf(y) = (\monom[][y])^{1/\degree}$. However this is a contradiction to $y\in\S\setminus\conv{\{\zerovec,\onevec\}}$ because observe that for any $x\ge\zerovec$, $\concf(x) = (\monom)^{1/\degree}$ if and only if $x_{1}=x_{2}=\dots=x_{n}$, which is equivalent to $x\in\conv{\{\zerovec,\onevec\}}$. Therefore $\S\cap\relint{\{\zerovec,\onevec\}}\neq\emptyset$ is necessary for the proposed upper bound to be tight.

Suppose  that $\S\cap\conv{\{\zerovec,\onevec\}}=\conv{\{\xi_{1}\onevec,\xi_{2}\onevec \}}$ for some $0\le\xi_{1}\le\xi_{2}\le 1$.  On $\relint{\{\zerovec,\onevec\}}$, the function $\concf(x)-\monom$ transforms to the univariate concave function $\xi - \xi^{\degree}$ for $\xi\in(0,1)$, which has a unique stationary point at $\tilde{\xi}=\degree^{1/(1-\degree)}$, giving us $\tilde{\xi} - \tilde{\xi}^{\degree} = \degree^{1/(1-\degree)} - \degree^{\degree/(1-\degree)} = {\xi^{\prime}}^{1/\degree}-\xi^{\prime}$, if  $\xi^{\prime} = \degree^{\frac{\degree}{1-\degree}}$. The function $\xi - \xi^{\degree}$ is increasing on $(0,\tilde{\xi})$ and decreasing on $(\tilde{\xi},1)$. By construction of $\fmin$ and $\fmax$, it follows that $0\le \fmin\le \xi_{1}^{\degree} \le \xi_{2}^{\degree} \le \fmax\le 1$. Therefore $\concf(x)-\monom = {\xi^{\prime}}^{1/\degree}-\xi^{\prime}$ for some $x\in\S$ if and only if $x={\xi^{\prime}}^{1/\degree}\onevec$ and $\xi_{1} \le {\xi^{\prime}}^{1/\degree} \le \xi_{2}$.
\end{proof}

The upper bound presented in Theorem~\ref{thm:concub} depends on the minimum and maximum values of the monomial over $\S$, which can be hard to compute for arbitrary $\S$, and not just on the degree of the monomial. However, an immediate consequence is that the constant $\erropt$, defined as $\erropt := (\degree-1)\,\degree^{\frac{\degree}{1-\degree}}$ in equation~\eqref{eq:C}, is a degree-dependent bound on the error from $\concf(x)$.
 
\begin{corollary}\label{corr:concub}
$\err{\graph[\S][\concf]} \le \erropt$, and this bound is tight if $\degree^{1/(1-\degree)}\onevec\in\S$  and only if $\fmin\le\degree^{\degree/(1-\degree)} \le \fmax$.
\end{corollary}
\begin{proof}
The function $\xi^{1/\degree}-\xi$ attains its maxima over $[0,1]$ uniquely at $\xi_{0}=\degree^{\degree/(1-\degree)}$. 
The definition of $\xi^{\prime}$ then gives us \[{\xi^{\prime}}^{1/\degree}-\xi^{\prime} \le \left(\degree^{\frac{\degree}{1-\degree}}\right)^{\frac{1}{\degree}} - \degree^{\frac{\degree}{1-\degree}} = \degree^{\frac{1}{1-\degree}} - \degree^{\frac{\degree}{1-\degree}} = (\degree-1)\degree^{\frac{\degree}{1-\degree}} = \erropt, \] and subsequently, Theorem~\ref{thm:concub} leads to $\erropt$ being an upper bound on $\concf(x) - \monom$. The uniqueness of the maxima of $\xi^{1/\degree}-\xi$ also implies that for $\erropt$ to be a tight bound, we must have $\xi^{\prime}=\degree^{\degree/(1-\degree)}$, which is equivalent to $\fmin\le\degree^{\degree/(1-\degree)} \le \fmax$. 
\end{proof}

Notice that the necessity of $\fmin\le\degree^{\degree/(1-\degree)} \le \fmax$ in the above corollary is not immediate from the statement of Theorem~\ref{thm:concub}. This can be explained as follows. Denote $\S\cap\conv{\{\zerovec,\onevec\}}=\conv{\{\xi_{1}\onevec,\xi_{2}\onevec \}}$ for some $0\le\xi_{1}\le\xi_{2}\le 1$. Since we showed  that $\erropt$ is an upper bound on ${\xi^{\prime}}^{1/\degree}-\xi^{\prime}$, Theorem~\ref{thm:concub} implies that if $\erropt$ is a tight bound then $\xi_{1} \le {\xi^{\prime}}^{1/\degree} \le \xi_{2}$. By construction, $\xi^{\prime}\in\{\fmin,\fmax,\degree^{\degree/(1-\degree)}\}$ and $(\fmin)^{1/\degree}\le \xi_{1} \le \xi_{2} \le (\fmax)^{1/\degree}$. So, by Theorem~\ref{thm:concub}, it is possible to have $\degree^{1/(1-\degree)} < (\fmin)^{1/\degree}$ or $\degree^{1/(1-\degree)} > (\fmax)^{1/\degree}$, if $\erropt$ is tight. However, Corollary~\ref{corr:concub} rules out this possibility. Furthermore, the condition $\fmin\le\degree^{\degree/(1-\degree)} \le \fmax$ is not sufficient to guarantee tightness of $\erropt$. The reason being that this condition does not enforce non-emptiness of $\S\cap\relint{\{\zerovec,\onevec\}}$, which we know to be necessary from Theorem~\ref{thm:concub}. 

If the minimum and maximum values of $\monom$ over $\S$ are low-enough and high-enough, respectively, as per Corollary~\ref{corr:concub}, then we have a precise characterization of when $\erropt$ is a tight bound on $\concf(x) - \monom$.

\begin{corollary}\label{corr:concub2}
For any $\S\subseteq\Hzone$ with $\fmin\le\degree^{\degree/(1-\degree)} \le \fmax$, the upper bound $\erropt$ on $\concf(x) - \monom$ is tight if and only if $\degree^{1/(1-\degree)}\onevec\in\S$. In particular, $\err{\graph[][\concenv{\f}{\Hzone}]} = \erropt$. 
\end{corollary}
\begin{proof}
The assumptions of $\fmin$ and $\fmax$ imply $\xi^{\prime} = \degree^{\degree/(1-\degree)}$ in Theorem~\ref{thm:concub}, thereby leading to the first claim. Since $\fmin[\Hzone]=0$, $\fmax[\Hzone]=1$, and $\degree^{1/(1-\degree)}\onevec\in\relint\{\zerovec,\onevec\}$, the second claim follows from the first part and $\concf = \concenv{\f}{\Hzone}$ from \eqref{eq:crama2}.
\end{proof}

For the simplex $\Deltaone := \conv{\{\onevec,\onevec-\onevec[1],\dots,\onevec-\onevec[n] \}}$,  clearly,  $\fmin=0,\fmax=1$ for any  $\S\supseteq\Deltaone$. \blue{This simplex can be described as $\Deltaone = \{x\mid \sum_{j=1}^{n}x_{j} \ge n-1, \ x\le \onevec  \}$. When $\degree = n$, i.e., multilinear monomial, it is easy to verify graphically that $n^{1/(1-n)} < 1 - 1/n$ so that the point $n^{1/(1-n)}\onevec$ does not belong to $\Deltaone$. However, the function $\degree^{1/(1-\degree)}$ being monotone in $\degree$, for large enough values of $\degree$, we have $\degree^{1/(1-\degree)} \ge 1 - 1/n$, as can be verified numerically, and consequently, $\degree^{1/(1-\degree)}\onevec\in\Deltaone$. Hence, the bound $\erropt$ from Corollary~\ref{corr:concub2} is tight for arbitrary $\S\supseteq\Deltaone$ when the monomial degree is large}. 

\subsubsection{Standard simplex}
For monomials considered over the standard $n$-simplex $\Delta_{n}$, we obtain a bound in Proposition~\ref{prop:concubsimpl} that is tight only for symmetric monomials. The proof of this result uses the following lemma which will be useful also in proving Theorem~\ref{thm:converr01} later in \textsection\ref{sec:converr}.

\begin{lemma}\label{lem:dineq}
$\degree^{(\degree-1)^{2}}>\degree^{\degree(\degree-2)} \ge (\degree-1)^{(\degree-1)^{2}}$ for all $\degree\ge 2$, and $\degree^{\degree(\degree-2)} = (\degree-1)^{(\degree-1)^{2}}$ if and only if $\degree=2$. 
\end{lemma}
\begin{proof}
Obviously $\degree^{(\degree-1)^{2}}=\degree^{\degree^{2}-2\degree+1}>\degree^{\degree(\degree-2)}$. Since $\degree/(\degree-1) = 1 + 1/(\degree-1)$, binomial expansion gives us
\[
\left( \frac{\degree}{\degree-1} \right)^{(\degree-1)^{2}} \;\ge\; 1 + \frac{(\degree-1)^{2}}{\degree-1} \;=\; \degree, \qquad \degree\ge 2.
\]
This is equivalent to $\degree^{\degree(\degree-2)} \ge (\degree-1)^{(\degree-1)^{2}}$. Clearly, equality holds for $\degree=2$. For $\degree\ge 3$, binomial expansion gives us
\[
\left(\frac{\degree}{\degree-1} \right)^{(\degree-1)^{2}} \;\ge\; 1 + \frac{(\degree-1)^{2}}{\degree-1} + \binom{(\degree-1)^{2}}{2}\frac{1}{(\degree-1)^{2}} \;=\; \degree + \frac{\degree(\degree-2)}{2} \;>\; \degree,\qquad\degree\ge 3,
\]
thereby leading to $\degree^{\degree(\degree-2)} > (\degree-1)^{(\degree-1)^{2}}$.
\end{proof}

\begin{proposition}\label{prop:concubsimpl}
\[ \err{\graph[\Delta_{n}][\concf]} \;\le\; \frac{(\monom[][\power])^{1/\degree}}{\degree} \, -\, \frac{\monom[][\power]}{\degree^{\degree}}, \] and this bound is tight if and only if $\power_{1}=\cdots=\power_{n}$.
\end{proposition}
\begin{proof}
$\fmin[\Delta_{n}]=0$ because $\zerovec\in\Delta_{n}$. The maximum value of $\monom$ over $\Delta_{n}$ is obviously attained in the relative interior of the face defined by the plane $\sum_{j=1}^{n}x_{j}= 1$. Solving the KKT system for $\fmax[\Delta_{n}] = \max_{\bx}\,\{\monom\mid \sum_{j=1}^{n}x_{j}= 1\}$ gives us $\fmax[\Delta_{n}] = \frac{\monom[][\power]}{\degree^{\degree}}$. For fixed integers $2 \le n \le \degree$, it is easy to argue that 
\[\max_{\power}\left\{\monom[][\power]\mid \power\in\posint^{n},\sum_{j=1}^{n}\power_{j}=\degree \right\} \;=\; (\degree-n+1)^{\degree-n+1},\] using  the convexity of $t\mapsto t\log{t}$ and the integrality of the polytope $\{\power\in\real^{n}_{\ge 1}\mid \sum_{j}\power_{j}=\degree  \}$. Therefore for fixed $\degree$, the maximum value of $\monom[][\power]$ is achieved with $n=2$ and is equal to $(\degree-1)^{\degree-1}$. Thus, $\fmax[\Delta_{n}] = \monom[][\power]/\degree^{\degree} \le (\degree-1)^{\degree-1}/\degree^{\degree} $. By Lemma~\ref{lem:dineq}, we have $\degree^{\degree(\degree-2)/(\degree-1)} \ge (\degree-1)^{\degree-1}$ and so 
\[\fmax[\Delta_{n}] \le \frac{\degree^{\degree(\degree-2)/(\degree-1)}}{\degree^{\degree}} = \degree^{\degree/(1-\degree)}.\] 
This implies that $\xi^{\prime} = \fmax[\Delta_{n}]$ in Theorem~\ref{thm:concub}, thereby giving us  the proposed upper bound on $\concf(x) - \monom$. Theorem~\ref{thm:concub} also tells us that this bound is tight if and only if $\frac{(\monom[][\power])^{1/\degree}}{\degree}\,\onevec \in \Delta_{n}$, which is equivalent to showing $\monom[][\power] \le (\degree/n)^{\degree}$. Observe the following.
\begin{claim}
$\monom[][\power] \ge (\degree/n)^{\degree}$ for $\power\ge\onevec$, with equality holding if and only if  $\power_{1}=\cdots=\power_{n}$.
\end{claim}
\begin{proof}[Proof of Claim]\renewcommand{\qedsymbol}{$\diamond$}
This inequality is obtained by applying Jensen's inequality to the convex function $t\in(0,\infty)\mapsto t\log{t}$  with the $n$ points being $t_{i}=\power_{i} \ \forall i$ and the convex combination weights being all equal to $1/n$. The equality condition is due to $t\log{t}$ being strictly convex.
\end{proof}
Therefore our bound is tight if and only if the monomial is symmetric.
\end{proof}


\subsection{Convex underestimator error}\label{sec:cvxerr}
We address the case of a simplex first because it is easy. 

\begin{proposition}\label{prop:cvxubsimpl}
Suppose that $\S$ is a $0\backslash 1$ polytope with $\onevec\notin \S$. Then $\cvxenv{\f}{\S}(\cdot)=0$. In particular, $\cvxenv{\f}{\Delta_{n}}(\cdot)=0$, and the error due to this envelope is equal to $\power^{\power}/\degree^{\degree}$.
\end{proposition}
\begin{proof}
Observe the following fact which is an immediate consequence of  applying Jensen's inequality to the definition of convex envelope: for a continuous function $\phi\colon X\mapsto[\phi_{0},\infty)$ for some finite $\phi_{0}$ and bounded polyhedral domain $X$, if  $\phi(v)=\phi_{0}$ for every vertex $v$ of $X$, then $\cvxenv{\phi}{X}(\cdot) = \phi_{0}$. Since $\f(x)\ge 0$ for $x\ge\zerovec$ and $\f(x)=0$ for $x\in\{0,1\}^{n}\setminus\{\onevec\}$, it follows from the assumption on $\S$ that $\cvxenv{\f}{\S}(\cdot)=0$. The standard $n$-simplex $\Delta_{n}$ satisfies the assumption on $\S$ and so the convex envelope over it is the zero function, thereby making the error equal to $\fmax[\Delta_{n}]$. This value  was argued in the proof of Proposition~\ref{prop:concubsimpl} to be equal to $\power^{\power}/\degree^{\degree}$.
\end{proof}

Hereafter, we let $\S$ be an arbitrary subset of $\Hzone$, with a special interest in $S=\Hzone$, or more generally $\S\supseteq\Deltaone[\lambda]$, where
\begin{subequations}\label{eq:delta1}
\begin{equation}
\Deltaone[\lambda] :=  \conv{\{\onevec,\onevec-\lambda_{1}\onevec[1],\dots,\onevec-\lambda_{n}\onevec[n]\}} = \left\{x\le\onevec\mid \sum_{j=1}^{n}\frac{x_{j}}{\lambda_{j}} \ge \sum_{j=1}^{n}\frac{1}{\lambda_{j}} - 1 \right\}, \ \; \zerovec<\lambda\le\onevec,
\end{equation}
is a $n$-simplex cornered at $\onevec$. For convenience, we write $\Deltaone[\onevec]$ simply as $\Deltaone$. The motivation for studying the case $\S\supseteq\Deltaone[\lambda]$ is clear from Proposition~\ref{prop:cvxubsimpl} which highlights the significance of the vertex $\onevec$ belonging to $\S$. Also note that the polytope $\Delta^{\zerovec}_{n}$, the complement of $\Deltaone$ defined as
\begin{equation}
 \Delta^{\zerovec}_{n} := \conv{\left(\{0,1\}^{n}\setminus\{\onevec\}\right)} = \left\{x\in\Hzone\mid \sum_{j=1}^{n}x_{j}\le n-1 \right\},
\end{equation}
\end{subequations}
is a $0\backslash 1$ polytope not containing $\onevec$. \blue{Note that $\Delta^{\zerovec}_{n}$ is \emph{not} the simplex cornered at $\zerovec$, which was defined in \textsection\ref{sec:intro} to be $\Delta_{n}$}. If $\Delta^{\zerovec}_{n}\subseteq\S\subseteq\Hzone$, $\cvxenv{\f}{\S}(x)=0$ for all $x\in\Delta^{\zerovec}_{n}$, and therefore one would be interested in finding strong convex underestimators of $\monom$ over $\S\setminus\Delta^{\zerovec}_{n}$. We will derive a piecewise  linear convex underestimator later in Proposition~\ref{prop:Prelax}.

\blue{
We begin by establishing an error bound in Theorem~\ref{thm:errenv01}. This bound does not have an explicit expression or formula, rather it is stated as the infimum of a certain function. However, it serves as a stepping stone towards deriving explicit error bounds in \textsection\ref{sec:explicit} that depend only on the degree of the polynomial, and hence towards proving our main result in \textsection\ref{sec:converr}.

\subsubsection{Implicit bound}
}

Unlike \textsection\ref{sec:concerr} where we calculate the error from a specific concave overestimator, here we consider a general convex underestimator defined as the pointwise supremum of a family of affine functions,
\begin{subequations}
\begin{equation}\label{eq:cvxf}
\cvxf(x) := \max\left\{0,\,\sup_{\beta\in\B}\;\sigma(\beta) + \sum_{j=1}^{n}\beta_{j}(x_{j}-1)\right\}, 
\end{equation}
for some nonempty (possibly countably infinite) set $\B\subseteq\onevec+\nonnegreal[n]$, where
\begin{equation}\label{eq:sigma}
\sigma(\beta) := \min_{\bx\in\S}\,\monom-\sum_{j=1}^{n}\beta_{j}(x_{j}-1)
\end{equation}
for each $\beta\in\B$ to ensure that the linear function $\sigma(\beta) + \sum_{j=1}^{n}\beta_{j}(x_{j}-1)$ underestimates and touches the graph of $\monom$. For finite $\B$, $\cvxf$ is a piecewise linear convex underestimator, otherwise $\cvxf$ could represent the convex envelope of $\monom$ over $\S$. The assumption of nonnegativity on $\beta$ is due to the fact that the gradient of $\monom$ at any point in $\nonnegreal[n]$ is a nonnegative vector. For convenience, we allow only positive $\beta$ and scale it greater than equal to 1 by assuming $\B\subseteq\onevec+\nonnegreal[n]$. The multilinear monomial with $\S=\Hzone$ would have $\cvxf(x) = \max\{0, 1 + \sum_{j=1}^{n}(x_{j}-1)\}$ (cf. \eqref{eq:crama}) with $\B=\{\onevec\}$ and $\sigma(\onevec)=1$. 

Denote $\degree[\beta]:= \sum_{j=1}^{n}\beta_{j}$. This gives us 
\begin{equation}\label{eq:sigma}
\sigma(\beta) = \degree[\beta]\, +\, \min_{\bx\in\S}\monom - \beta^{\top}\bx.
\end{equation}
\end{subequations}

Towards proving our main error bound in terms of only the degree of the monomial, we first obtain in Theorem~\ref{thm:errenv01} a error bound that depends on $\sigma(\beta)$'s. We make some remarks on $\sigma(\beta)$ here. An explicit formula for $\sigma(\beta)$ for arbitrary $\S$ seems hard and the function is expected to be nonconvex ($\sigma(\beta)$ is a translate of the negative of the Fenchel conjugate of $\monom$). However, it is possible to find bounds on it, which we state next. 

\begin{proposition}\label{prop:sigmad}
We have the following for $\sigma(\beta)$ when $\S\subseteq\Hzone$:
\begin{enumerate}
\item $0\le\sigma(\beta)<\degree[\beta]$.
\item If $\S=\Hzone$, then \blue{$0 \le \sigma(\beta) \le 1$}. 
\end{enumerate}
Let \blue{$(\cdot)$ be the permutation that sorts $\beta$ as} $\beta_{(1)}\ge \beta_{(2)} \ge \cdots \ge \beta_{(n)}$. 
\begin{enumerate}[resume]
\item If $\Delta^{\zerovec}_{n}\cap\Deltaone\subseteq\S\subseteq\Delta^{\zerovec}_{n}$, then $\sigma(\beta) = \beta_{(n)}$.
\end{enumerate}
\end{proposition}

The proof is moved to Appendix~\ref{sec:app1}. \blue{The case $\S=\Delta_{n}$ is not covered in the above proposition since the error over $\Delta_{n}$ was already dealt with in Proposition~\ref{prop:cvxubsimpl} and hence we would have no use of the bounds on $\sigma(\beta)$ in this case.} 

To establish \blue{an} upper bound on $\monom - \cvxf(x)$, we define the following constants for every linear underestimator $\sigma(\beta) + \sum_{j}\beta_{j}(x_{j}-1)$:
\begin{equation}\label{eq:Cdef}
\begin{split}
\ratio := \max_{j}\frac{\beta_{j}}{\bb_{j}}, 
\quad &\erropt[\beta,\bb] := \left(1 - \frac{\sigma(\beta)}{\degree[\beta]} \right)^{\frac{\degree[\beta]}{\ratio}}, \quad \text{for } \bb\ge\onevec. 
\end{split}
\end{equation}
It is clear that \blue{$\ratio < \degree[\beta]$} and so $ 0 < \ratio/\degree[\beta] < 1$. Since $0\le\sigma(\beta) < \degree[\beta]$ by Proposition~\ref{prop:sigmad}, we have $0 < \erropt[\beta,\bb] \le 1$. For any $\beta\in\B$, $\ratio[][\cdot]$ is a nonincreasing function and so $\erropt[\beta,\cdot]$ is also a nonincreasing function:
\begin{equation}\label{eq:Cinc}
0 < \erropt[\beta,\bb^{\prime}]\le \erropt[\beta,\bb]\le 1, \quad \onevec\le\bb\le\bb^{\prime}.
\end{equation}
We do not know how $\erropt[\cdot,\bb]$ behaves. The significance of the scalar $\erropt[\beta,\bb]$ is as follows.

{
\renewcommand{\g}{\varphi_{\beta,\bb}}
\begin{lemma}\label{lem:g}
Define $\g\colon t \mapsto \degree[\beta] - \sigma(\beta) + t - \degree[\beta]t^{\frac{\ratio}{\degree[\beta]}}$. For $\bb,\beta$ with $\bb\ngtr\beta$, \[\erropt[\beta,\bb] = \max_{t\in[0,1]} \min\{t, \g(t)\}.\]
\end{lemma}
\begin{proof}
Since $\ratio/\degree[\beta]\in(0,1)$, $\g$ is convex over $[0,\infty)$. It is decreasing only over $[0,t_{0}]$, where $t_{0} := \ratio^{\degree[\beta]/(\degree[\beta]-\ratio)}$ is the unique stationary point of $\g$. Note that $\g(0) = \degree[\beta] - \sigma(\beta) > 0$ and observe that $\erropt[\beta,\bb]$, which lies in $(0,1]$,  is the unique fixed point of $\g$ on $\real$. Hence $\min\{t,\g(t)\} = t$ if and only if $t\in[0,\erropt[\beta]]$. The assumption $\bb\ngtr\beta$ is equivalent to $\ratio \ge 1$. Therefore $\degree[\beta]/\ratio \le \degree[\beta]$. We claim that \[ \left(1 - \frac{\sigma(\beta)}{\degree[\beta]}\right)^{\frac{\degree[\beta]}{\ratio}} \ge \left(1 - \frac{\sigma(\beta)}{\degree[\beta]}\right)^{\degree[\beta]} \ge 1-\sigma(\beta). \] The first inequality is obvious whereas the second is due to the monotonicity of the function $\sigma\mapsto (1 - \frac{\sigma}{\degree[\beta]})^{\degree[\beta]} + \sigma - 1$ on $[0,\degree[\beta]]$. Thus we have argued that $\g(1) \le \g(\erropt[\beta,\bb])$. Now the monotone behavior of $\g$ on $[t_{0},\infty)$ means that $t_{0} > \erropt[\beta,\bb]$ because otherwise we would have the contradiction $\g(1) > \g(\erropt[\beta,\bb])$. This implies that the maximum value of $\min\{t,\g(t)\}$ on the $[0,1]$ interval occurs at $t=\erropt[\beta,\bb]$ and, since this is a fixed point, it is equal to $\erropt[\beta,\bb]$.
\end{proof}

Since we need $\bb\ngtr\beta$ in the above lemma and forthcoming results, define 
\begin{equation}\label{eq:K}
\begin{split}
\K &:= \{\bb\in\nonnegreal[n]\mid \onevec \le \bb \le \power, \: \bb \ngtr \beta \ \, \forall \beta\in\B\}\\
&= \{\bb\in\nonnegreal[n]\mid \onevec \le \bb \le \power, \: \bb_{j}\le \min_{\beta\in\B}\beta_{j} \ \text{ for some } j\}.
\end{split}
\end{equation}
The assumption $\B\subseteq\onevec + \nonnegreal[n]$ makes it obvious that $\onevec\in\K$. The structure of $\g$ discussed in the proof of Lemma~\ref{lem:g} implies the following claim. 

\begin{lemma}\label{lem:g2}
For every $\bb\in\K$, \[\max_{t\in[0,1]}\inf_{\beta\in\B}\min\{t,\g(t)\} \;=\; \inf_{\beta\in\B}\max_{t\in[0,1]}\min\{t,\g(t)\} \;=\;\inf_{\beta\in\B}\erropt[\beta,\bb]. \]
\end{lemma}


We are now ready to state our upper bound on error from the convex underestimator $\cvxf$.

\begin{theorem}\label{thm:errenv01}
\[
\err{\graph[\S][\cvxf]} \:\le\: \inf_{\beta\in\B}\erropt[\beta,\bbmax],
\] where $\bbmax$ is a maximal element of $\K$ under the partial order $\le$. In particular, if there exists some $j$ such that $\power_{j} \le \beta_{j}$ for all $\beta\in\B$, then 
\[
\err{\graph[\S][\cvxf]} \;\le\; \erropt[\beta^{\ast},\power] \;:=\; \inf_{\beta\in\B}\erropt[\beta,\power],\] and this bound is tight only if $\beta^{\ast}=\power$ and is attained only at the point $\erropt[\power,\power]^{1/\degree}\onevec\in\relint{\{\zerovec,\onevec\}}$.
\end{theorem}
\begin{proof}
Choose some $\bb\in\K$. For every $x\in\Hzone$ and $j$, $\beta_{j}/\bb_{j} \le \ratio$ gives us $x_{j}^{\beta_{j}/\bb_{j}} \ge x_{j}^{\ratio}\ge0$ and $\bb_{j}\le\power_{j}$ gives us $x_{j}^{\bb_{j}}\ge x_{j}^{\power_{j}}\ge0$.  Thus
\begin{subequations}
\begin{equation}\label{eq:xbeta}
\monom[\beta] = \prod_{j=1}^{n}\left(x_{j}^{\frac{\beta_{j}}{\bb_{j}}} \right)^{\bb_{j}} \ge \prod_{j=1}^{n}\left(x_{j}^{\ratio} \right)^{\bb_{j}} = \prod_{j=1}^{n}\left(x_{j}^{\bb_{j}} \right)^{\ratio} \ge \prod_{j=1}^{n}\left(x_{j}^{\power_{j}} \right)^{\ratio}   = \left(\monom \right)^{\ratio}. 
\end{equation}
The generalized arithmetic-geometric means inequality tells us that $\sum_{j=1}^{n}\beta_{j}x_{j} \ge \degree[\beta](x^{\beta} )^{1/\degree[\beta]}$, which combined with \eqref{eq:xbeta} leads to $\sum_{j=1}^{n}\beta_{j}x_{j} \ge \degree[\beta]\left(\monom \right)^{\ratio/\degree[\beta]}$. Therefore 
\begin{eqnarray*}
\monom - \cvxf(x) &=&  \min\left\{\monom,\, \inf_{\beta\in\B}\monom - \sigma(\beta) + \degree[\beta] - \sum_{j=1}^{n}\beta_{j}x_{j} \right\} \\
&\le& \min\left\{\monom,\, \inf_{\beta\in\B}\monom - \sigma(\beta) + \degree[\beta] - \degree[\beta]\left(\monom \right)^{\frac{\ratio}{\degree[\beta]}} \right\} \\
&=& \inf_{\beta\in\B}\,\min\left\{\monom,\, \monom - \sigma(\beta) + \degree[\beta] - \degree[\beta]\left(\monom \right)^{\frac{\ratio}{\degree[\beta]}} \right\}, 
\end{eqnarray*}
which leads to
\begin{equation}\label{eq:bdcvx2}
\max_{x\in\S}\,\monom - \cvxf(x) \;\le\; \max_{x\in\S} \inf_{\beta\in\B}\, \min\left\{\monom,\,\monom - \sigma(\beta) + \degree[\beta] - \degree[\beta]\left(\monom \right)^{\frac{\ratio}{\degree[\beta]}} \right\}.
\end{equation}
Since $f(x)=\monom$ is a continuous function with minimum and maximum values $\fmin,\fmax\in[0,1]$ on $\S$, the intermediate value theorem implies that \eqref{eq:bdcvx2} transforms to
\begin{equation}\label{eq:lastbdcvx}
\max_{x\in\S}\,\monom - \cvxf(x) \;\le\; \max_{t\in[0,1]}\,\inf_{\beta\in\B}\,\min\{t,\g(t)\},
\end{equation}
where $\g(t) = \degree[\beta] - \sigma(\beta) + t - \degree[\beta]t^{\frac{\ratio}{\degree[\beta]}}$ as in Lemma~\ref{lem:g}. 
Lemma~\ref{lem:g2} leads to $\max_{x\in\S}\monom - \cvxf(x) \le \inf_{\beta}\erropt[\beta,\bb]$. Since $\bb$ was arbitrarily chosen in $\K$ and we know from \eqref{eq:Cinc} that $\erropt[\beta,\cdot]$ is a nonincreasing function for every $\beta$, we may set $\bb$ equal to a maximal $\bbmax$ to obtain $\max_{x\in\S}\monom - \cvxf(x) \le \inf_{\beta\in\B}\erropt[\beta,\bbmax]$. If $\power_{j}\le \min_{\beta\in\B}\beta_{j}$ for some $j$, then $\power$ is the unique maximal element in $\K$ and setting $\bbmax=\power$ yields the upper bound $\inf_{\beta\in\B}\erropt[\beta,\power]$.
\end{subequations}

The bound $\erropt[\beta^{\ast},\power]$ is tight if and only if there is equality throughout in \eqref{eq:xbeta} with $\bb=\power,\beta=\beta^{\ast}$, and in the means inequality $\sum_{i=1}^{n}\beta^{\ast}_{i}x_{i} \ge \degree[\beta^{\ast}](x^{\beta^{\ast}} )^{1/\degree[\beta^{\ast}]}$. Equation \eqref{eq:xbeta} is an equality if and only if $\bb=\power=\beta^{\ast}$, implying that $\beta^{\ast}=\power$ is a necessary condition for tightness. The means inequality is an equality if and only if $x_{1}=x_{2}=\cdots=x_{n}$ and hence the bound can be attained only at $\erropt[\power,\power]^{1/\degree}\onevec$. 
\end{proof}

\begin{remark}
We will show in Proposition~\ref{prop:Prelax} that $\sigma(\power) = 1$, implying that $\erropt[\power,\power] = (1-1/\degree)^{\degree}$, which is exactly the constant $\erroptcvx$ defined in \eqref{eq:C}, and therefore the above bound is attained only at $(1-1/\degree)\onevec$.
\end{remark}
}

Any polyhedral relaxation of the epigraph of $\monom$ can be encoded by the set $\B$ in equation~\eqref{eq:cvxf}. Hence Theorem~\ref{thm:errenv01} yields an upper bound on the error from any polyhedral relaxation that is chosen apriori. Since we do not know the behavior of $\erropt[\cdot,\bb]$, a analytic expression for the infimum in Theorem~\ref{thm:errenv01} does not seem possible in general. Even if $\B$ is finite, $\erropt[\beta,\bb]$ requires the computation of $\sigma(\beta)$, which we know to be hard in general. However, one may derive upper bounds on the error using the lower bounds on $\sigma(\beta)$ from Proposition~\ref{prop:sigmad}. Note though that this does not help for $\S=\Hzone$ because the lower bound of $0$ on $\sigma(\beta)$ gives a trivial upper bound of 1 on the error. 

We use the bound in Theorem~\ref{thm:errenv01} to derive a degree-dependent bound on the convex envelope error. To do so, let us view this upper bound  from a different perspective. By construction of $\erropt[\beta,\bb]$, in order to obtain  a smaller error bound, we would intuitively want to pick $\B$ such that it contains only those $\beta$ that make $\sigma(\beta)$ to be as high as possible. For $\S=\Hzone$, or more generally $\S$ containing $\onevec$, we know the highest that $\sigma(\beta)$ can be is 1. Hence we could do the following reverse construction --- instead of choosing a set $\B$ and then computing $\sigma(\beta)$ for each $\beta\in\B$ as done before, we could fix $\sigma(\beta)=1$ and find the values of $\beta\ge\onevec$ that enable $1 + \sum_{j=1}^{n}\beta_{j}(x_{j}-1)$ to be a valid linear underestimator (cf. equation~\eqref{eq:cvxf}) to $\monom$ over $\S$. This would alleviate the issue of having to compute $\sigma(\beta)$ for $\erropt[\beta,\bb]$ and could possibly lead to simpler and explicit error bounds that depend only on exponent $\power$ and degree $\degree$. We follow this path for the rest of this section. Note also that the convex envelope of the multilinear monomial $\multilin$ over $\Hzone$ is $\max\{0,1+\sum_{j}(x_{j}-1)\}$, meaning that there is only one $\beta$, the vector $\onevec$, with $\sigma(\onevec)=1$. Thus our forthcoming derivation implies the error from the convex envelope of a multilinear monomial over $\Hzone$.


\newcommand{\nd}[1]{\mathcal{ND}(#1)}

\renewcommand{\B}{\mathcal{B}_{1}}

\subsubsection{Explicit bounds}\label{sec:explicit}
Denote \[\linear(x) := 1 + \sum_{j=1}^{n}\beta_{j}(x_{j}-1), \quad \beta\ge\onevec.\] This linear function is exact at $x=\onevec$: $\linear(\onevec)=1=\f(\onevec)$. The convex underestimator on $\monom$ is 
\begin{equation}\label{eq:cvxf2}
\cvxf[\g](x) = \max\left\{0, \sup_{\beta\in\B}\linear(x)\right\}, \ \; \text{ where } \B := \{\beta\ge\onevec\mid \linear(x)\le\monom \ \forall x\in\S \}.
\end{equation}
$\B$ is a closed convex set\footnote{It does not seem that $\B$ will be a polyhedron even for $\S=\Hzone$. Since general monomials are not vertex-extendable over $\Hzone$, it is not clear whether the validity of $\linear$ over the entire box can be certified by checking at only a finite number of points.}, due to linearity of $\linear(x)$ in $\beta$ for fixed $x$, and it represents all the linear functions that are exact at $x=\onevec$ and underestimate $\monom$ everywhere on $\S\subseteq\Hzone$. Clearly, $\beta\le\beta^{\prime}$ implies $\linear(x) \ge \linear[\beta^{\prime}](x)$ for all $x\in\Hzone$, and so $\beta\in\B$ implies $\beta^{\prime}\in\B$. But then we could simply delete such a $\beta^{\prime}$ from $\B$ without affecting the supremum in $\cvxf[\g]$. Hence we define the nondominated subset of $\B$ to be the following: 
\begin{equation}\label{eq:B1}
\nd{\B} := \{\beta\in\B\mid \nexists\, \onevec\le\beta^{\prime}\lneqq\beta \text{ s.t. } \linear[\beta^{\prime}](x) \le \monom \ \forall x\in\S   \},
\end{equation}
so that 
\begin{equation}\label{eq:cvxf3}
\cvxf[\g](x) = \max\left\{0, \sup_{\beta\in\nd{\B}}\linear(x)\right\}. 
\end{equation}

A strong error bound from $\cvxf[\g]$ would obviously depend on the elements in $\nd{\B}$ (cf. Theorem~\ref{thm:errenv01}), making it important to obtain a (partial) characterization of $\B$ and $\nd{\B}$ based on the structure of $\S$. We mention two cases where $\nd{\B}$ is easily seen to be equal to $\{\onevec\}$, the most trivial value. 
\begin{description}
\item[Multilinear over $\Hzone$.] Here $\power=\onevec,\S=\Hzone$ and equation~\eqref{eq:crama} tells us $\onevec\in\B$, and therefore $\nd{\B}=\{\onevec\}$.

We will generalize this in Proposition~\ref{prop:Prelax} by showing that $\nd{\B}=\{\power\}$ when $\S\supseteq\Deltaone[\lambda]$.

\item[Subsets of $\Delta^{\zerovec}_{n}$.] Here $\power$ is arbitrary and $\S\subseteq\Delta^{\zerovec}_{n}=\conv{(\{0,1\}^{n}\setminus\{\onevec\})}$. We know that $\linear$ is valid to $\S$ if and only if $\sigma(\beta)\ge 1$, where $\sigma(\beta) = \min_{x\in\S}\monom - \sum_{j}\beta_{j}(x_{j}-1)$. Clearly $\linear$ is valid to $\S$ if it is valid to $\Delta^{\zerovec}_{n}$.  We argued in Proposition~\ref{prop:sigmad} that $\sigma(\beta)=\beta_{(n)}$ for $\Delta^{\zerovec}_{n}$ and since $\beta\ge\onevec$ by assumption, it follows that $\linear$ is valid to $\S$ for all $\beta\ge\onevec$. Therefore $\nd{\B}=\{\onevec\}$.
\end{description}

For an arbitrary integer exponent $\power$ and $\S\nsubseteq\Delta^{\zerovec}_{n}$, it is not at all obvious what the set $\B$ should be. Note that this includes the case of a monomial over $\S=\Hzone$. As a generalization of the multilinear case, is it true that $\power\in\B$? The function $\linear[\power](\cdot)$ is Taylor's first-order approximation of $\monom$ at the point $x=\onevec$. Having $\power\in\B$ would mean that the gradient inequality at $x=\onevec$ holds true, which is not at all obvious since $\monom$ is a nonconvex function. We show in Proposition~\ref{prop:Prelax} that $\power\in\B$ is always true, regardless of $\S$, and in fact construct a $\beta\le\power$ with $\beta\in\B$, so that $\power\notin\nd{\B}$ in general. This $\beta$ depends on $\S$ and is constructed by taking projections of $\S$ onto each coordinate. We also present some conditions under which $\nd{\B}$ can be (partially) characterized.
 
The following technical lemma will be useful. It is proved in Appendix~\ref{sec:app1}.

{
\renewcommand{\chi}{\sigma}
\renewcommand{\beta}{\lambda}
\renewcommand{\xi}{t}
\begin{lemma}\label{lem:exp1}
Let $\beta_{1}\in\posint,\beta_{2}\ge 1$. Consider the univariate polynomial $\phi(\chi) := (1-\chi)^{\lambda_{1}}+\lambda_{2}\chi-1$ which has a trivial root at $0$.

\begin{enumerate}
\item If $\beta_{2}\ge\beta_{1}$, $\phi(\chi) > 0$ for all $\chi\in(0,1]$.
\end{enumerate}
For $\beta_{2}<\beta_{1}$, 
\begin{enumerate}[resume]
\item $\phi$ has exactly one root in $(0,1]$, denoted $\chi^{\ast}$, and $\chi^{\ast} > 1-(\beta_{2}/\beta_{1})^{\frac{1}{\beta_{1}-1}}$.

\item $\phi(\chi) < 0$ for all $\chi\in(0,\chi^{\ast})$ and $\phi(\chi) > 0$ for all $\chi\in(\chi^{\ast},1]$.

\item $(1-\chi)^{\beta_{1}} > 1-\beta\chi$ for all $\beta\in(\beta_{2},\infty),\chi\in[\chi^{\ast},1]$, and $(1-\chi)^{\beta_{1}} < 1-\beta\chi$ for all $\beta\in[1,\beta_{2}),\chi\in[0,\chi^{\ast})$ 
\end{enumerate}
Finally, there is a root in $(1,\infty)$ if and only if $\beta_{1}$ is odd, and there is a root in $(-\infty,0)$ if and only if $\beta_{2}>\beta_{1}$.
\end{lemma}

\begin{remark}\label{rem:root}
\blue{Finding an analytic expression for the root $\chi^{\ast}$ seems difficult, and an algebraic root may not even exist, as can be verified using computational algebra software for the polynomial $\phi(\sigma) = (1-\sigma)^{6} + 3\sigma - 1$, whose roots are in bijection to that of $\sigma^{6} - 3\sigma + 2$ \blue{under the mapping $\sigma\mapsto 1-\sigma$}. However, our forthcoming analysis circumvents this issue since it does not depend on the exact value of $\chi^{\ast}$.}
\end{remark}
}

We also need to introduce some notation. For every $i$, denote the projection of $\S$ onto the $x_{i}$-subspace by \[\proj_{x_{i}}\S := [1-\sigma^{1}_{i},1-\sigma^{2}_{i}],\quad \text{ for some } 0 \le \sigma^{2}_{i}\le\sigma^{1}_{i}\le 1,\] and define 
\begin{equation}\label{eq:bbdef}
\gamma_{i} := \begin{cases} 
\displaystyle\frac{1 - (1-\sigma^{2}_{i})^{\power_{i}}}{\sigma^{2}_{i}} & \text{ if } \sigma^{2}_{i} > 0,\medskip\\
\power_{i} &   \text{ if } \sigma^{2}_{i} = 0,
\end{cases}\qquad  i=1,\dots,n.
\end{equation}
This $\gamma$ is exactly the $\gamma$ from the statement of Theorem~\ref{thm:converr01} in \textsection\ref{sec:mainmonom}. Note that if $\S\cap E_{i}\neq\emptyset$, $\S\cap E_{j}\neq\emptyset$ for distinct $i,j$, then $\sigma^{2}=\zerovec$.
\begin{lemma}\label{lem:bb}
$1\le\gamma_{i}<\power_{i}$ for every $i$ with $\sigma^{2}_{i}>0$. Hence $\gamma=\power$ if and only if $\sigma^{2}=\zerovec$.
\end{lemma}
\begin{proof}
$\gamma_{i}\ge1$ is obvious due to $\sigma^{2}_{i}\in(0,1)$ and $\power_{i}\ge1$. Since $\power_{i}\in\posint$, we have $\frac{1-\chi^{\power_{i}}}{1-\chi} = 1+\chi+\chi^{2}+\dots+\chi^{\power_{i}-1}$, making $\frac{1-\chi^{\power_{i}}}{1-\chi}$ an increasing function on $[0,1]$. Hence, by complementing to $\sigma=1-\chi$, $\frac{1 - (1-\sigma)^{\power_{i}}}{\sigma}$ is a decreasing function on $[0,1]$. L'H\^opital's rule gives $\lim_{\sigma\to 0}\frac{1 - (1-\sigma)^{\power_{i}}}{\sigma} = \power_{i}$.
\end{proof}

\begin{proposition}\label{prop:Prelax}
We have the following:
\begin{enumerate}
\item $\gamma,\power\in\B$.
\end{enumerate}
Consider any $\beta\ge\onevec$ and suppose $I := \{i\mid \S\cap \relint E_{i}\neq\emptyset, \beta_{i}\le\power_{i}\}$ is nonempty. For  $i\in I$  denote $1-\tau^{2}_{i} = \max\{x_{i}\mid x\in \S\cap  E_{i}\}$.  
\begin{enumerate}[resume]
\item $\beta\in\B$ only if $\power_{i}(1-\tau^{2}_{i})^{\power_{i}-1} \le \beta_{i} \le \power_{i}$ for $i\in I$ with $\tau^{2}_{i}>\sigma^{2}_{i}$, and $\gamma_{i} \le \beta_{i} \le \power_{i}$ for $i\in I$ with $\tau^{2}_{i}=\sigma^{2}_{i}$.  
\item Suppose $\onevec\in\S$. Then $\beta\in\B$ only if $\beta_{i} = \power_{i}$ for all $i\in I$.
\end{enumerate}
Finally, 
\begin{enumerate}[resume]
\item If $\S\supseteq\Deltaone[\lambda]:=\conv{(\cup_{i=1}^{n}\{\onevec-\lambda_{i}\onevec[i] \})}$ for some $\zerovec<\lambda\le\onevec$, then $\nd{\B}=\{\power\}$.
\end{enumerate}
\end{proposition}
\begin{proof}\renewcommand{\epsilon}{\sigma}
(1) Observe that showing $\linear(x)\le\monom$ for all $x\in\S$ is equivalent to showing $\linear(x)\le\monom$ for all $x\in\S$ such that $x>\zerovec,x\neq\onevec$. Indeed, $\linear(x)$ is exact at $x=\onevec$ and for any $x\in E_{0}$, $x_{i}=0$ implies that $\linear(x) = 1-\beta_{i}+\sum_{j\neq i}\beta_{j}(x_{j}-1)$ which is nonpositive due to $\beta\ge\onevec$ and $x\in\Hzone$. Therefore to show $\gamma\in\B$, we prove $\linear[\gamma](x)\le\monom$ for every $x\in \S,x>\zerovec,x\neq\onevec$. 

Consider such an $x$ and let $k = |\{i\mid 0 < x_{i} < 1\}|$. Assume wlog that $x_{i}=1-\sigma_{i}$ for $i =1,\dots,k$ with $\sigma_{i}\in[\sigma^{2}_{i},\sigma^{1}_{i}], \sigma_{i}\in(0,1)$, and $x_{i}=1$ for $i \ge k+1$. We must show that \[\prod_{i=1}^{k}(1-\sigma_{i})^{\power_{i}}\ge 1 - \sum_{i=1}^{k}\gamma_{i}\sigma_{i}.\] We argue this inequality by induction on $k$. Take $k=1$. We obtain $(1-\epsilon_{1})^{\power_{1}}\ge 1 - \gamma_{1}\epsilon_{1}$ from the following claim.
 
\begin{claim}\label{claim:proof1}
For any $i$ and $\sigma\in[\sigma^{2}_{i},1]$, we have $(1-\sigma)^{\power_{i}} \ge 1-\beta_{i}\sigma$ for all $\beta_{i}\ge\gamma_{i}$.
\end{claim}
\begin{proof}[Proof of Claim]\renewcommand{\qedsymbol}{$\diamond$}
If $\sigma^{2}_{i}=0$, then $\gamma_{i}=\power_{i}$ and applying the first item in Lemma~\ref{lem:exp1} with $\lambda_{1}=\power_{i}$ and $\lambda_{2}=\beta_{i}$ tells us $(1-\epsilon)^{\power_{i}}\ge 1 - \beta_{i}\epsilon$ for all $\beta_{i}\ge\gamma_{i}$. Otherwise $\sigma^{2}_{i}>0$ and Lemma~\ref{lem:bb} allow\blue{s} us to apply Lemma~\ref{lem:exp1} with $\lambda_{1}=\power_{i}$ and $\lambda_{2}=\gamma_{i}$. It is readily seen from the construction of $\gamma$ in \eqref{eq:bbdef} that $\sigma^{2}_{i}$ is a root of $\phi(\omega) = (1-\omega)^{\power_{i}} + \gamma_{i}\omega-1$ and by the second item of Lemma~\ref{lem:exp1}, it is the unique root in $(0,1]$. Now $\blue{\sigma}\in[\sigma^{2}_{i},1]$ and the fourth item of Lemma~\ref{lem:exp1} yield $(1-\epsilon)^{\power_{i}}\ge 1 - \beta_{i}\epsilon$ for all $\beta_{i}\ge\gamma_{i}$.
\end{proof}

Assume that the inequality is true for $k\ge 1$ and let us argue it for $k+1$. 
The induction hypothesis gives us \[
\prod_{i=1}^{k}(1-\epsilon_{i})^{\power_{i}} - \gamma_{k+1}\epsilon_{k+1} \ge 1 - \sum_{i=1}^{k}\gamma_{i}\sigma_{i} - \gamma_{k+1}\epsilon_{k+1} = 1 + \sum_{i=1}^{n}\gamma_{i}(x_{i}-1).\] 
Let $\prod_{i=1}^{k}(1-\epsilon_{i})^{\power_{i}} = 1+\chi$ for some $\chi \in (-1, 0)$; such a $\chi$  exists because $\epsilon_{i}\in(0,1)$ for $i=1,\dots,k$. Hence, the induction hypothesis becomes \[1+\chi - \gamma_{k+1}\sigma_{k+1} \ge 1 + \sum_{i=1}^{n}\gamma_{i}(x_{i}-1).\] Now, \[ \monom = \prod_{i=1}^{k+1}(1-\epsilon_{i})^{\power_{i}} = (1+\chi)(1-\epsilon_{k+1})^{\power_{k+1}} \ge (1+\chi)(1-\gamma_{k+1}\epsilon_{k+1}),\]
where the inequality is by applying Claim~\ref{claim:proof1} to $i=k+1$, and using $1+\chi > 0$. Since $(1+\chi)(1-\gamma_{k+1}\epsilon_{k+1}) = 1 + \chi -\gamma_{k+1}\epsilon_{k+1} -\gamma_{k+1}\epsilon_{k+1}\chi$ and $\chi < 0, \gamma_{k+1},\epsilon_{k+1} > 0$, we have 
\begin{eqnarray}
\monom > 1 + \chi -\gamma_{k+1}\epsilon_{k+1} \ge 1 + \sum_{i=1}^{n}\gamma_{i}(x_{i}-1), \label{eq:strictcvx}
\end{eqnarray}
where $\ge$ is from the induction hypothesis. This finishes our inductive proof for showing $\gamma\in\B$. Thus every $x\in\S$ with $|\{i\mid x_{i}\in(0,1)\} |\ge 2$ has $\linear[\gamma](x) < \monom$. The closedness of $\B$ under monotonicity and $\gamma\le\power$ give us $\power\in\B$.

(2) Choose some $\blue{i}\in I$. If $\beta_{i}=\power_{i}$, then there is nothing to prove because $\power_{i}\ge\gamma_{i}$ and $\tau^{2}_{i}\in[0,1]$. So assume $\beta_{i}<\power_{i}$. Consider a point $\bar{x}\in\S\cap\relint E_{i}$, which can be written as $\bar{x}_{i}=1-\tau$ and $\bar{x}_{j}=1\ \forall j\neq i$, where $\tau=\tau^{2}_{i}$ if $\tau^{2}_{\blue{i}}>0$, otherwise $\tau$ is a small positive real. Note that $\monom[][\bar{x}] = (1-\tau^{2}_{i})^{\power_{i}}$ and $\linear(\bar{x}) = 1-\beta_{i}\tau$. The second and third items of Lemma~\ref{lem:exp1} with $\lambda_{1}=\power_{i}, \lambda_{2}=\beta_{i}$ tell us that $(1-\tau)^{\power_{i}} <  1-\beta_{i}\tau$ if $\tau \le 1 - (\beta_{i}/\power_{i})^{\frac{1}{\power_{i}-1}}$. This means that $\tau^{2}_{i} \ge 1 - (\beta_{i}/\power_{i})^{\frac{1}{\power_{i}-1}}$, which rearranges to $\beta_{i} \ge \power_{i}(1-\tau^{2}_{i})^{\power_{i}-1}$, is necessary for $\linear$ to be a valid linear underestimator.

(3) It is easy to see that the convexity of $\S$ makes $\onevec\in\S$ equivalent to $\tau^{2}_{i}=0$ for all $i\in I$. We also have $\onevec\in\S$ implying $\sigma^{2}=\zerovec$. Therefore $\gamma=\power$. Now (2) gives us $\beta_{i}=\power_{i}$ for $i\in I$.



(4) The assumption $\S\supseteq\Deltaone[\lambda]$ implies $\S\cap\relint E_{i}\neq\emptyset$ for all $i$, $\onevec\in\S$, $\tau^{2}=\sigma^{2}=\zerovec$ and hence $\gamma=\power$. The claim then follows from (3).
\end{proof}


\begin{remark}
Due to the functions $h_{1}(t) = \power_{i}(1-t)^{\power_{i}-1}$ and $h_{2}(t) = h_{1}(t) - (1 - (1-t)^{\power_{i}})/t$ being nonincreasing \blue{and nonpositive, respectively}, over $[0,1]$, it follows that $\power_{i}(1-\tau^{2}_{i})^{\power_{i}-1} \le \gamma_{i}$ in Proposition~\ref{prop:Prelax}, meaning that the lower bound on $\beta_{i}$ with $\tau^{2}_{i} > \sigma^{2}_{i}$ is weaker than the lower bound on $\beta_{i}$ with $\tau^{2}_{i} = \sigma^{2}_{i}$. This happens because while arguing this part, we used a lower bound on the root of $(1-\sigma)^{\lambda_{1}}+\lambda_{2}\sigma-1$ in $(0,1]$ from Lemma~\ref{lem:exp1}, since finding a analytic expression for the root seems difficult (cf. Remark~\ref{rem:root}). Therefore if $I\setminus I^{\prime}\neq\emptyset$, then there is no guarantee that $\gamma$ is a nondominated point in $\B$.
\end{remark}

\begin{remark}
The second item in Proposition~\ref{prop:Prelax} indicates that a tight lower bound on a valid $\beta$ can get arbitrarily close to $\power$. 
\end{remark}




The vector $\gamma$ in \eqref{eq:bbdef} can be constructed only when projections of $\S$ are readily available or can be computed quickly. When these projections are difficult to compute, we could use the first claim of Proposition~\ref{prop:Prelax} telling us that $\linear[\power]$ is a underestimator of $\monom$. The last item in this proposition provides a clean and simple expression for $\cvxf[\g]$ in \eqref{eq:cvxf3}.

The preceding results on $\B$ and $\nd{\B}$, combined with Theorem~\ref{thm:errenv01}, imply explicit bounds on the error from the convex underestimator $\cvxf[\g]$. Recall the constants from \eqref{eq:Cdef}. Denoting $\erropt[\beta,\beta]$ simply as $\erropt[\beta]$, we have for $\linear[\power]$ and $\linear[\gamma]$, respectively,: \[
\erropt[\power] = \left(1 - \frac{1}{\degree} \right)^{\degree} = \erroptcvx, \qquad \erropt[\gamma] = \left(1 - \frac{1}{\degree[\gamma]} \right)^{\degree[\gamma]},
\] 
where we recall that $\erroptcvx$ was defined in \eqref{eq:C} and $\degree[\gamma]=\sum_{j=1}^{n}\gamma_{j}$. 

\begin{corollary}\label{corr:cvxdegbound}
$\err{\graph[\S][{\cvxf[\g]}]} \le \erropt[\gamma] \le \erroptcvx$, and equality holds throughout if $\zerovec\in\S$ and $\S\supset\Deltaone[\lambda]$ for some $\zerovec<\lambda\le\onevec$.
\end{corollary}
\begin{proof}
We first observe that $\max_{x\in\S}\,\monom - \max\{0,\linear[\gamma](x)\} \le \erropt[\gamma]$. This is obtained by applying Theorem~\ref{thm:errenv01} with $\cvxf$ replaced by $\max\{0,\linear[\gamma]\}$ and noting that $\gamma$ is a maximal element of $\K[\{\gamma\}]$. Since $\gamma\in\B$ by Proposition~\ref{prop:Prelax}, $\cvxf[\g](\cdot)\ge \max\{0,\linear[\gamma](\cdot)\}$ and hence $\max_{x\in\S}\,\monom - \cvxf[\g](x) \:\le\: \erropt[\gamma]$. Since $t\ln{(1 - \frac{1}{t})}$ is concave increasing over $[2,\infty)$ and $\gamma\le\power$ by construction, we get $\erropt[\gamma] \le \erropt[\power]$. If $\S\supseteq\Deltaone[\lambda]$, then $\gamma=\power$ and the last claim in Proposition~\ref{prop:Prelax} tells us $\nd{\B}=\{\power\}$ and $\cvxf[\g](x)=\max\{0,\linear[\power](x)\}$. Now recall Theorem~\ref{thm:errenv01}. We have $\beta^{\ast}=\power$ due to $\nd{\B}=\{\power\}$. This theorem tells us that the bound on $\max_{x\in\S}\,\monom - \cvxf[\g](x)$ can be attained only at $\erropt[\power]\onevec$. The assumptions $\zerovec\in\S$ and $\Deltaone[\lambda]\subset\S$ lead to $\conv\{\zerovec,\onevec\}\subset\S$ and therefore $\erropt[\power]\onevec\in\S$.
\end{proof}

A direct implication is a tight bound on the error of the convex envelope of a multilinear monomial considered over $\Hzone$.
\begin{corollary}\label{corr:converr01multi}
We have $\max_{x\in\Hzone}\,\monom - \max\{0,\linear[\power](x)\} = \erroptcvx$. In particular, for a multilinear monomial, $\err{\graph[][\cvxenv{\f[m]}{\Hzone}]} = (1 - \frac{1}{n})^{n}$.
\end{corollary}
\begin{proof}
Since $\S=\Hzone\supset\Deltaone$, the last item in Proposition~\ref{prop:Prelax} tells us $\cvxf[\g](x) = \max\{0,\linear[\power](x)\}$ and then the first equality follows immediately from Corollary~\ref{corr:cvxdegbound}. For a multilinear monomial, equation~\eqref{eq:crama} gives us $\cvxenv{\f[m]}{\Hzone}(x) = \max\{0,\linear[\onevec](x)\}$. The claimed error follows by using $\power=\onevec$ in the expression for $\erroptcvx$. 
\end{proof}

\subsection{Convex hull error} \label{sec:converr}
\begin{proof}[\textbf{Proof of Theorem~\ref{thm:converr01}}]
Since $\concenv{\f}{\S}(x)\le\concf(x)$ for $x\in\Hzone$, the upper bound of $\erropt$ on $\concenv{\f}{\S}(x) - \monom$ is due to $\err{\graph[\S][\concf]} \le \erropt$ from Corollary~\ref{corr:concub}. Similarly the upper bounds on $\monom - \cvxenv{\f}{\S}(x)$ are due to $\cvxf[\g](\cdot)\le \cvxenv{\f}{\S}(\cdot)$ and Corollary~\ref{corr:cvxdegbound}. By Observation~\ref{obs:errmaxenv}, we then have that $\err{\conv{\graphmonom}} \le \max\{\erropt,\erroptcvx \}$. To show this error is upper bounded by $\erropt$, we argue the following. 
\begin{claim}
$\erroptcvx \le \erropt$ for $\degree\ge 2$ and equality holds if and only if $\degree=2$. 
\end{claim}
\begin{proof}[Proof of Claim]\renewcommand{\qedsymbol}{$\diamond$}
The two constants are  $\erroptcvx=\erropt[\power] = (1-1/\degree)^{\degree}$ and $\erropt = (1 - 1/\degree)\degree^{1/(1-\degree)}$. Therefore the following equivalence holds:
\begin{align*}
\erropt  \ge \erroptcvx 
&\iff \left( \frac{1}{\degree}\right)^{\frac{1}{\degree-1}} \ge \left(1 - \frac{1}{\degree} \right)^{\degree-1} 
 &\iff \; 
\degree \le \left(\frac{\degree}{\degree-1} \right)^{(\degree-1)^{2}} \;\iff\; (\degree-1)^{(\degree-1)^{2}} \le \degree^{\degree(\degree-2)}.
\end{align*}
Lemma~\ref{lem:dineq} proves the last inequality and that it holds at equality only when $\degree=2$. 
\end{proof}
Thus we have $\err{\conv{\graphmonom}} \le \erropt$ for any $\S\subseteq\Hzone$. 

If $\zerovec,\onevec\in\S$,  then setting $t_{1}=0,t_{2}=1$ in Lemma~\ref{lem:lowerbd} yields the critical point to be $\xi^{\prime}=(1/\degree)^{1/(\degree-1)}$ so that \[\phi(\xi^{\prime}) = \xi^{\prime} - {\xi^{\prime}}^{\degree} = \xi^{\prime}(1 - {\xi^{\prime}}^{\degree-1}) = (1/\degree)^{\frac{1}{\degree-1}}( 1 - 1/\degree ) = \erropt.\] Therefore the convex hull error and the concave envelope error are lower bounded by $\erropt$, making each of them equal to $\erropt$. 
\end{proof}

{
\renewcommand{\P}[1][\B]{P_{#1}}
\renewcommand{\B}{\mathcal{B}}
The arguments used in proving Theorem~\ref{thm:converr01} also imply that a family of convex relaxations of $\graphmonom$ has error equal to $\erropt$. Recall the convex underestimator $\cvxf$ from \eqref{eq:cvxf} for any $\B\subseteq\onevec + \nonnegreal[n]$  and consider the convex relaxation \[\P := \{(x,w)\in\Hzone\times\real\mid \cvxf(x) \le w \le \concf(x)\}.\] Note that $x$ is not restricted to be in $\S$ here. Assume $\power\in\B$. Also assume $\onevec\in\S$ so that $\sigma(\beta) \le 1$ for every $\beta\in\B$, as per  Proposition~\ref{prop:sigmad}. We claim that 

\begin{proposition}\label{prop:Perr}
$\err{\P}=\erropt$. 
\end{proposition}
\begin{proof}
The proof of $\err{\P}\le\erropt$ is the same as that in Theorem~\ref{thm:converr01}, along with using the assumption $\power\in\B$ to get $\max_{x\in\S} \monom - \cvxf(x) \le \max_{x\in\S}\monom - \max\{0,\linear[\power](x)\} = \erroptcvx$. Tightness of this bound is obtained by applying Lemma~\ref{lem:lowerbd} and Remark~\ref{rem:lowerbd} after noting that $(\zerovec,0),(\onevec,1)\in\P$. The point $(\zerovec,0)$ belongs to $\P$ because $\concf(\zerovec)=0$, and $\cvxf(\zerovec) = \max\{0,\sup_{\beta\in\B}\sigma(\beta)-\degree[\beta] \}$, which is equal to $0$ since Proposition~\ref{prop:sigmad} states that $\sigma(\beta) < \degree[\beta]$. The point $(\onevec,1)$ belongs to $\P$ because $\concf(\onevec)=1$, and $\cvxf(\onevec) = \max\{0,\sup_{\beta\in\B}\sigma(\beta)\}$, which is less than equal to 1 due to $\onevec\in\S$.
\end{proof}
}

The next proof is that of the error bounds over a simplex.

\begin{proof}[\textbf{Proof of Theorem~\ref{thm:simplex}}]
The concave envelope error bound is from Proposition~\ref{prop:concubsimpl} and the fact that $\concenv{\f}{\Delta_{n}} \le \concf$. The convex envelope error bound was observed in Proposition~\ref{prop:cvxubsimpl}.  To upper bound $\err{\conv{\graph[\Delta_{n}][\f]}}$, we note that 
\[
\frac{(\monom[][\power])^{1/\degree}}{\degree} \, -\, \frac{\monom[][\power]}{\degree^{\degree}} \;\ge\; \frac{\monom[][\power]}{\degree^{\degree}} \iff \left(\frac{\monom[][\power]}{\degree^{\degree}}\right)^{\frac{\degree-1}{\degree}} \;\le\; \frac{1}{2} \iff 2^{\frac{1}{\degree-1}} \;\le\; \frac{\degree}{\monom[{\power/\degree}][\power]}. 
\] 
Denoting $\power_{(n)} = \max_{i}\power_{i}$, we have $\monom[{\power/\degree}][\power] \le \power_{(n)}^{\sum_{i}\power_{i}/\degree} = \power_{(n)}$. Thus it suffices to show that $\degree/\power_{(n)} \ge 2^{\frac{1}{\degree-1}}$, equivalently, $(\degree/\power_{(n)})^{\degree-1} \ge 2$. Since $\power_{n}\le \degree-1$ due to $n\ge 2$,  
\[ \left(\frac{\degree}{\power_{(n)}}\right)^{\degree-1} \ge \left(\frac{\degree}{\degree-1}\right)^{\degree-1} = \left(1 + \frac{1}{\degree-1} \right)^{\degree-1} \ge 1 + \frac{\degree-1}{\degree-1} = 2,
\]
\blue{where the last inequality is from binomial expansion}.
\end{proof}

{\renewcommand{\Honer}{\Hzone[\frac{1}{r}][1]}
We end by mentioning that for $\S=\Honer$, or equivalently for $\S=\Hzone[1][r]$ upto scaling,  our upper bounds on the convex hull error are the same as those in Theorem~\ref{thm:converr01} whereas a lower bound can be obtained by setting $t_{1}=1/r, t_{2}=1$ in Lemma~\ref{lem:lowerbd}. However these bounds are not tight, which is not all that surprising since we do not know the exact form of the envelopes of a general monomial over $\Honer$. In \textsection\ref{sec:polyerr2}, we consider a multilinear monomial over $\Hzone[1][r]$ and use the explicit characterization of its envelopes to derive tight error bounds. It so happens that  in the multilinear case, the lower bound from Lemma~\ref{lem:lowerbd} with $t_{1}=1,t_{2}=r$ seems to be the convex hull error, a claim that is verified empirically for random $r,n$ and shown to be true for every $r>1$ as $n\to\infty$.
}


\subsection{Comparison with another error bound}\label{sec:compare}
For the problem of optimizing $p\in\real[x]_{m}$ over $\S = \Hzone$: $\zopt[\blue{\Hzone}] = \min\{p(x)\mid x\in\Hzone \}$, \citet{de2010error} present a LP and a SDP relaxation of $\zopt[\blue{\Hzone}]$ based on two different positivstellensatz and also give a common error bound for these relaxations. Their bound is \citep[Theorem 1.4]{de2010error}: \[
\zopt[\blue{\Hzone}] - \tilde{z}^{\delta}_{\blue{\Hzone}} \le \frac{L(p)}{\delta}\binom{m+1}{3}n^{m},
\] 
where $\tilde{z}^{\delta}_{\Hzone}$ is either of their two relaxations, $\delta\ge m$ is an integer with $n\delta$ being a degree bound on polynomials in the positivstellensatz, and \[L(p) = \max_{\power}\,\abs{c_{\power}}\frac{\prod_{j}\power_{j}!}{\degree[\power]!}.\] As $\delta\to\infty$, the two relaxations converge to $\zopt[\blue{\Hzone}]$ (the SDP relaxation has finite convergence). Corollary~\ref{corr:L} states that the monomial convexification approach would  yield a error bound, \blue{as per our analysis},  of $\zopt[\blue{\Hzone}] - \zmon[\blue{\Hzone}] \le L^{\prime}(p)\binom{n+m}{n}$ for $L^{\prime}(p)$ defined in \eqref{eq:L}. This bound was  weakened subsequently in Corollary~\ref{corr:L2} for ease of computation.

{
\renewcommand{\degree}{m}
\blue{
We note that for the LP and SDP relaxations to provide a better worst case guarantee, the degrees of the polynomials considered in the respective positivstellensatz must grow cubic in the degree of $p(x)$.}


\begin{proposition}
For $p\in\real[x]_{m}$ with $c_{\power}=0,\pm 1$, and fixed $n$, the worst case error bound from $\tilde{z}^{\delta}_{\Hzone}$ is better than the worst case error bound from $\zmon[\Hzone]$ only if $\delta = {\Omega}(m^{3})$.
\end{proposition}
\begin{proof}
The assumption $c_{\power}=0,\pm 1$ implies $L(p) = \max_{\power\colon \abs{c_{\power}}=1}\,\frac{\prod_{j}\power_{j}!}{\abs{\power}!}$, which is lower bounded by $\frac{1}{m!}$. We have  $L^{\prime}(p) \le \erropt =  (1 - \frac{1}{m})m^{\frac{1}{(1-m)}}$ from Corollary~\ref{corr:L2}. Therefore, the LP and SDP relaxations of  \citep{de2010error} give better error  bounds than monomial convexification only if 
\[
\delta \ge \hat{\delta} := \frac{\binom{m+1}{3}n^{m}}{m!\erropt\binom{n+m}{n}} = \frac{m^{2}(m+1)}{6 m^{\frac{1}{1-m}} (1+\frac{m}{n})\cdots(1+\frac{1}{n})} = {\Omega}(m^{3}) \text{ for fixed $n$}. \qedhere
\]
\end{proof}
} 

\section{Multilinear monomial}\label{sec:polyerr2}
Here we consider a multilinear monomial $\f[m](x) = \multilin$ over either a box with constant ratio or a symmetric box. 
Since these boxes are simple scalings of $\Honer$ and $\Hnegoo$, respectively, and our error measure $\err{\cdot}$ scales as noted in Observation~\ref{obs:errscale}, we henceforth restrict our attention to only $\Honer$ and $\Hnegoo$. As in \textsection\ref{sec:err01}, the convex hull error is computed by bounding the convex and concave envelope errors separately. 

\subsection{Box with constant ratio} \label{sec:err1r}
\begin{proposition}[\citep{tawarmalani2013explicit,benson2004concave}]\label{prop:convoner}
\[ 
\concenv{\f[m]}{\Honer}(x) \;=\; \left[\min_{\sigma\in\Sigma_{n}}\,\sum_{j=1}^n r^{j-1}x_{\sigma(j)}\right] \,-\, \sum_{j=1}^{n-1} r^{j}, \] where $\Sigma_{n}$ is the set of all permutations of $\{1,\dots,n\}$, and 
\[
\cvxenv{\f[m]}{\Honer}(x) \;=\; \max_{i=1,\dots,n}\, r^{i-1}\left(\sum_{j=1}^n x_j \,-\, (n-i)\,-\,r(i-1)\right).
\]
\end{proposition}
\begin{proof}
To obtain $\concenv{\f[m]}{\Honer}$, we simply substitute $\L_{j}=1,\U_{j}=r$ in \citep[Theorem 1]{benson2004concave} which states $\concenv{\f[m]}{\H}$ for arbitrary $\L,\U$ with $\L\ge\zerovec$. The convex envelope can be derived from \citep[Theorem 4.6]{tawarmalani2013explicit}. This theorem gives a piecewise linear function with $n$ pieces as the convex envelope of a function $g(y)\colon\Hzone\mapsto\real$ when $g$ is convex-extendable from $\{0,1\}^{n}$ and there exists a convex function $\rho\colon\nonnegreal\mapsto\real$ such that $g(y) = \rho(\sum_{j=1}^{n}y_{j})$ for every $y\in\{0,1\}^{n}$. Consider $\multilin$. Writing $x_{j} = 1+(r-1)y_{j}$, the multilinear term becomes $g(y) = \prod_{j=1}^{n}(1+(r-1)y_j)$ for $y\in\Hzone$. Since this $g(y)$ is a multilinear function of $y$, it is convex-extendable from $\{0,1\}^{n}$. Furthermore, for $y\in\{0,1\}^{n}$, $g(y) =  r^{\sum_{j=1}^{n}y_j}$, and obviously $r^{(\cdot)}$ is convex over $\nonnegreal$. Therefore, the convex envelope formula follows from \citep[Theorem 4.6]{tawarmalani2013explicit}.
\end{proof}

Applying a straightforward scaling argument, similar to the one used for the $\Hzone$ box at the beginning of \textsection\ref{sec:err01}, gives us the convex hull of $\graph$ when $\L_{i}\U_{i}>0$ for all $i$ and for some $r>0$, $\U_{i}/\L_{i} = r$ for all $i$ with $\L_{i}>0$ and $\L_{i}/\U_{i}=r$ for all $i$ with $\U_{i}<0$. We omit the details.

Before proving Theorem~\ref{thm:converr1r} which claims that $\D$ and $\E$ are the maximum envelope errors for $\multilin$ over $\Honer$, we provide some background on these two constants. The value $\E$ is obtained by applying Lemma~\ref{lem:lowerbd}: set $t_{1}=1, t_{2}=r$ to get $\xi^{\prime} = (\frac{r^{n}-1}{n})^{\frac{1}{(n-1)}}(r-1)^{\frac{n}{(1-n)}} - \frac{1}{(r-1)}$ and $\phi(\xi^{\prime})$, upon simplification, becomes equal to $\E$. There is no simple explicit closed form formula for $\D$. 
However, $\D$ can be bounded as follows. After replacing $t=i/n$, the formula for $\D$ requires solving an integer program:\blue{
\[
\D = \max\{\psi(t) \mid t \in \{1/n,2/n,\dots,1 \} \}, \quad \text{ where } \psi(t) = (1 + (r-1)t)^{n} - r^{nt}.
\]
Note that $\psi$ is a difference of two convex increasing functions $\psi_{1}$ and $\psi_{2}$. After separating the maximizations over $\psi_{1}$ and $\psi_{2}$, we obtain the trivial upper bound $\D \le r^{n}-r$. But this bound can be very weak. A tighter bound can be derived by considering the continuous relaxation of the problem: \[
\D\le \max\{\psi(t)\mid t \in [0,1]\}.
\] 
Since $\psi$ is differentiable with $\psi(0) = \psi(1) = 0$, by Rolle's theorem, there exists at least one stationary point of $\psi$ in $[0,1]$.} Based on these stationary points, we can say the following.

\begin{proposition}\label{prop:D}
\blue{Let $t^{\ast} = \min\{t\in[0,1]\mid \psi^{\prime}(t)=0\}$ be the smallest stationary point of $\psi$ on $[0,1]$, and $t^{\ast\ast}$ be the global maxima of $\psi$ on $[0,1]$.} If $t^{\ast}\ge \frac{n-1}{n}$, then $\D = (1 + \frac{n-1}{n}(r-1))^{n} - r^{n-1}$, otherwise \blue{if $t^{\ast\ast} \le \frac{n-1}{n}$}, then $\D \le r^n \left(\frac{\ln r}{r-1}\right)^{\frac{n}{n-1}}-r^{n-1}$, otherwise \blue{$\D \le r^{\frac{n^{2}}{n-1}}\left(\frac{\ln r}{r-1}\right)^{\frac{n}{n-1}} \,-\, r^{n}$}.
\end{proposition}
\blue{The proof is in Appendix~\ref{sec:app1}. Obviously, $t^{\ast} \le t^{\ast\ast}$. We conjecture that $t^{\ast} = t^{\ast\ast}$.} 

We now prove our main result in this section.

\blue{\subsubsection{Proof of Theorem~\ref{thm:converr1r}}}
\begin{proof}
We only prove the maximum errors for the envelopes, the formula for $\err{\conv{\graph[\Honer]}}$ follows subsequently from Observation~\ref{obs:errmaxenv}. Consider the concave envelope first. We noted earlier that the value $\E$ comes from applying Lemma~\ref{lem:lowerbd} with $t_{1}=1, t_{2}=r$. Hence to prove that the maximum concave envelope error is equal to $\E$, it suffices to argue that there exists a point in $\relint \{\onevec,r\onevec \}$ which maximizes this error. Suppose, for sake of contradiction, that this is not the case. Since $\concenv{\f[m]}{\Honer}(\onevec) = \f[m](\onevec)$ and $\concenv{\f[m]}{\Honer}(r\onevec) = \f[m](r\onevec)$, we know that these two points do not maximize the error. Then our assumption means that for every maximizer $x^{\ast}$ there exists some index $i$ such that $x^{\ast}_{(i)} < x^{\ast}_{(i+1)}$, where $(\cdot)$ is the permutation \blue{that permutes variables as} $x^{\ast}_{(1)} \le x^{\ast}_{(2)} \le \cdots \le x^{\ast}_{(n)}$. Since $r > 1$, for every $x\in\Honer$, the minimum over $\Sigma_{n}$ in the expression for $\concenv{\f[m]}{\Honer}$, \blue{which is given in Proposition~\ref{prop:convoner}}, occurs at a permutation $\sigma$ such that $x_{\sigma(1)}\ge  x_{\sigma(2)}\ge\cdots \ge x_{\sigma(n)}$. Therefore, $\concenv{\f[m]}{\Honer}(x) = \sum_{j=1}^{n}r^{n-j}x_{(j)} - \sum_{j=1}^{n-1} r^{j}$. In particular, $\concenv{\f[m]}{\Honer}(x^{\ast}) = \sum_{j=1}^{n}r^{n-j}x^{\ast}_{(j)} - \sum_{j=1}^{n-1} r^{j}$, and the maximum error is $z^{\ast} = \sum_{j=1}^{n}r^{n-j}x^{\ast}_{(j)} - \prod_{j=1}^{n}x^{\ast}_{j} - \sum_{j=1}^{n-1} r^{j}$. Now consider two points $\hat{x}$ and $\tilde{x}$ obtained from $x^{\ast}$ by setting, respectively, $\hat{x}_{(i)} = x^{\ast}_{(i+1)}$ and $\tilde{x}_{(i+1)} = x^{\ast}_{(i)}$. Since the error at these points cannot be larger than $z^{\ast}$, we have $r^{n-i}(x^{\ast}_{(i+1)} - x^{\ast}_{(i)}) \le (x^{\ast}_{(i+1)} - x^{\ast}_{(i)})\prod_{j\neq i}x_{(j)}$ and $r^{(n-i-1)}(x^{\ast}_{(i)} - x^{\ast}_{(i+1)}) \le (x^{\ast}_{(i)} - x^{\ast}_{(i+1)})\prod_{j\neq i+1}x_{(j)}$, and consequently, $r^{n-i}-\prod_{j\neq \blue{i}} x^{\ast}_{(j)}\le 0$ and $r^{n-i-1}-\prod_{j\neq i+1} x^{\ast}_{(j)}\ge 0$. Hence
\[
r^{n-i}\;\le\; r^{n-i}x^{\ast}_{(i)} \;\le\; \prod_{j=1}^{n}x^{\ast}_{j} \;\le\; r^{n-i-1}x^{\ast}_{(i+1)}\;\le\; r^{n-i}.
\]
Equality holds in above if and only if $x^{\ast}_{(i)}=1$ and $x^{\ast}_{(i+1)}=r$. Therefore $x^{\ast}_{(1)}=\dots=x^{\ast}_{(i)}=1$, $x^{\ast}_{(i+1)}=\dots=x^{\ast}_{(n)}=r$, but at such a point, the error is zero due to \[\concenv{\f[m]}{\Honer}(x^{\ast}) \;=\; \sum_{j=1}^{i}r^{n-j} + \sum_{j=i+1}^{n}r^{n+1-j} - \sum_{j=1}^{n-1}r^{j} \;=\; r^{n-i}\; =\; \f[m](x^{\ast}).\] Thus we have reached a contradiction to $x^{\ast}$ being a maximizer. Hence it must be that the error is maximized on $\relint\{\onevec,r\onevec \}$.

\renewcommand{\phi}{\widehat{\varphi}}
Now consider the convex envelope. We follow similar steps as in the proof of Theorem~\ref{thm:errenv01}.
\[ 
\begin{split}
\multilin \,-\, \cvxenv{\f[m]}{\Honer}(x) &\;=\; \multilin \,-\, \max_{i=1,\dots,n} r^{i-1}\left(\sum_{j=1}^n x_j - (n-i)-r(i-1)\right)\\
&\;=\; \min_{i=1,\dots,n} \left\{ \multilin - r^{i-1}\left(\sum_{j=1}^n x_j - (n-i)-r(i-1)\right) \right\} \\
&\;\le\; \min_{i=1,\dots,n} \left\{ \multilin - r^{i-1}\left(n\sqrt[\leftroot{-1}\uproot{10}n]{\multilin} - (n-i)-r(i-1)\right) \right\},
\end{split}
\] where we employ the arithmetic-geometric means inequality. By regarding $\sqrt[n]{\multilin}$ as a scalar variable $t$, we get 
\[
\max_{x\in\Honer}\multilin \,-\, \cvxenv{\f[m]}{\Honer}(x) \;\le\; \max_{t\in[1,r]}\min_{i=1,\dots,n} \phi_{i}(t),
\] where $\phi_{i}\colon t \mapsto t^{n} - nr^{i-1}t + r^{i-1}[n-r + i(r-1)]$ is a convex function on $[1,r]$. Therefore we have to find the maximum value of the pointwise minimum function $\min_{i}\phi_{i}(t)$ on the interval $[1,r]$ and it is apparent that this maximum value is attained at a breakpoint of the function, i.e., at a $t^{\ast}$ such that $\phi_{i}(t^{\ast})=\phi_{i+1}(t^{\ast})$ for some $1\le i\le n-1$. For any $1\le i\le n-1$, solving for $t$ in $\phi_{i}(t) = \phi_{i+1}(t)$ means that we must find $t$ satisfying $t^{n} - r^{i-1}nt + r^{i-1}[n-r + i(r-1)] = t^{n} - r^{i}nt + r^{i}[n - i - 1 + ir]$, which upon canceling and rearranging terms leads to $r^{i-1}(r-1)nt = - r^{i-1}(n-i) - 2ir^{i} + nr^{i} + ir^{i+1}$. Therefore $(r-1)nt = -(n-i) - 2ir + nr + ir^{2} = n(r-1) +i(r-1)^{2}$ and hence, $t = 1 + i(r-1)/n$. Substituting this breakpoint $t$ into $\phi_{i}$ yields \[\phi_{i}(t) = \left(1 + \frac{i}{n}(r-1) \right)^{n} - nr^{i-1}\left(1 + \frac{i}{n}(r-1) \right) + r^{i-1}[n-r + i(r-1)] = \left(1 + \frac{i}{n}(r-1) \right)^{n} - r^{i}. \] The maximum, with respect to $i=1,\dots,n-1$, over all such values is the maximum of $\min_{i}\phi_{i}(t)$ and notice that this maximum over $i$ is exactly the constant $\D$. Hence $\D$ is a upper bound on the convex envelope error. This bound is tight because the means inequality is an equality when all the $x_{j}$'s are equal to each other, and hence this error is attained on $\relint{\{\onevec,r\onevec\}}$.
\end{proof}

\subsubsection{Comparing $\D$ and $\E$}
We conjecture that $\D\le \E$ for every $r,n$, which would imply that the convex hull error is equal to $\E$. Although we were unable to prove this in general due to the extremely complicated forms for $\E$, and more specifically, for $\D$, we ran some simulations, graphed in Figure~\ref{fig:compare}, to support our claim. For every $n=2,\dots,100$ and $r\in\{1.01,1.2,1.5,2,3,5,10\}$, we computed the ratio $\D/\E$ and plotted it in Figure~\ref{fig:ConvConc}. We also plotted the ratio between the error of the relaxed convex envelope (the relaxation is obtained by taking the maximum over $i\in\{1,n\}$ in the expression for $\cvxenv{\f[m]}{\Honer}$) and $\E$, see Figure~\ref{fig:RConvConc}. As can be seen in these figures, the ratios are never larger than 1, thereby establishing a strong empirical basis in support of our conjecture that the error from the concave envelope dominates that from the convex envelope, and possibly even from the relaxed convex envelope.

\begin{figure}[h]
\centering
\subfigure[Concave and Convex envelopes]{
\includegraphics[scale=0.52]{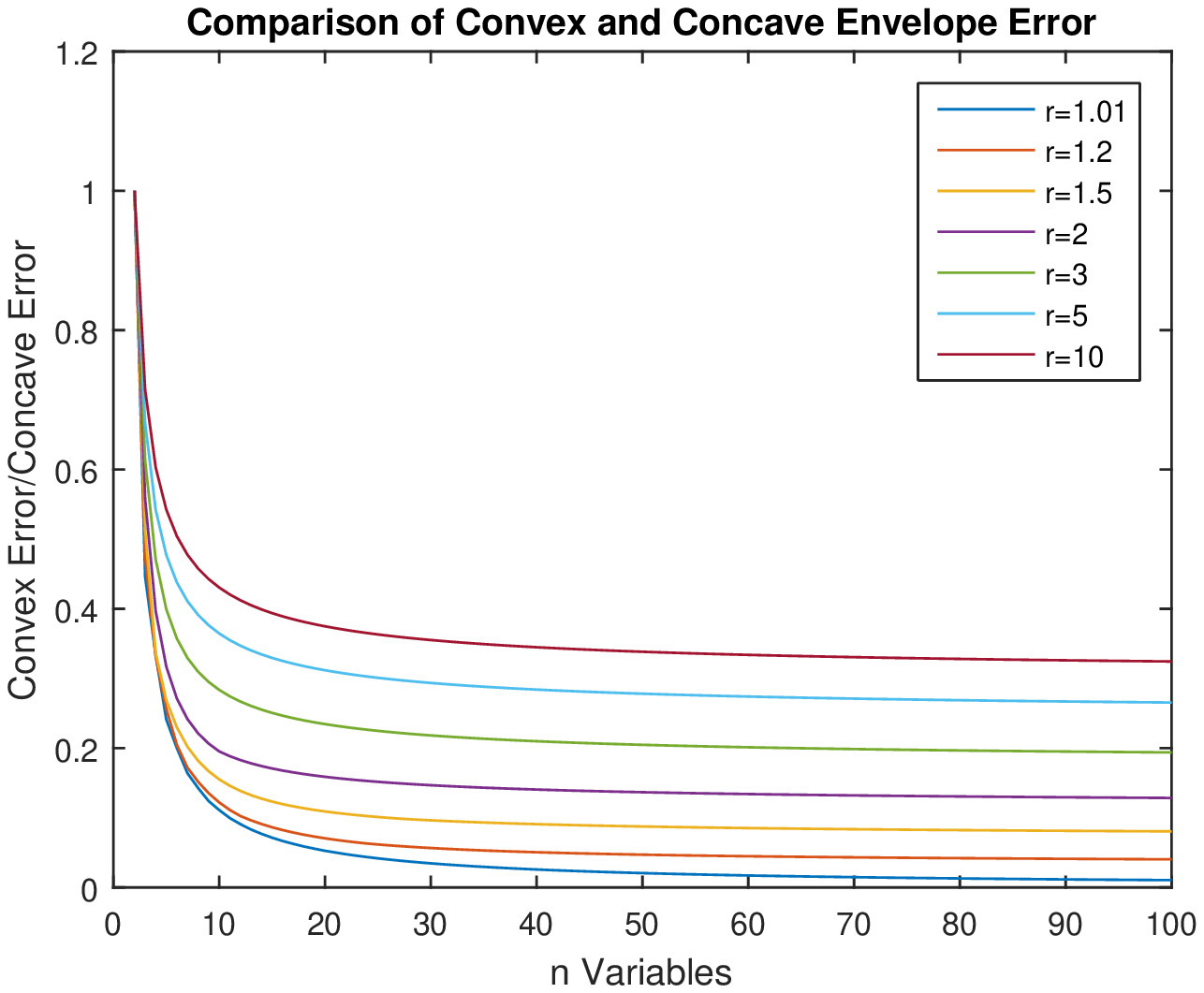}
\label{fig:ConvConc}
}
\subfigure[Concave and Relaxed Convex envelopes]{
\includegraphics[scale=0.52]{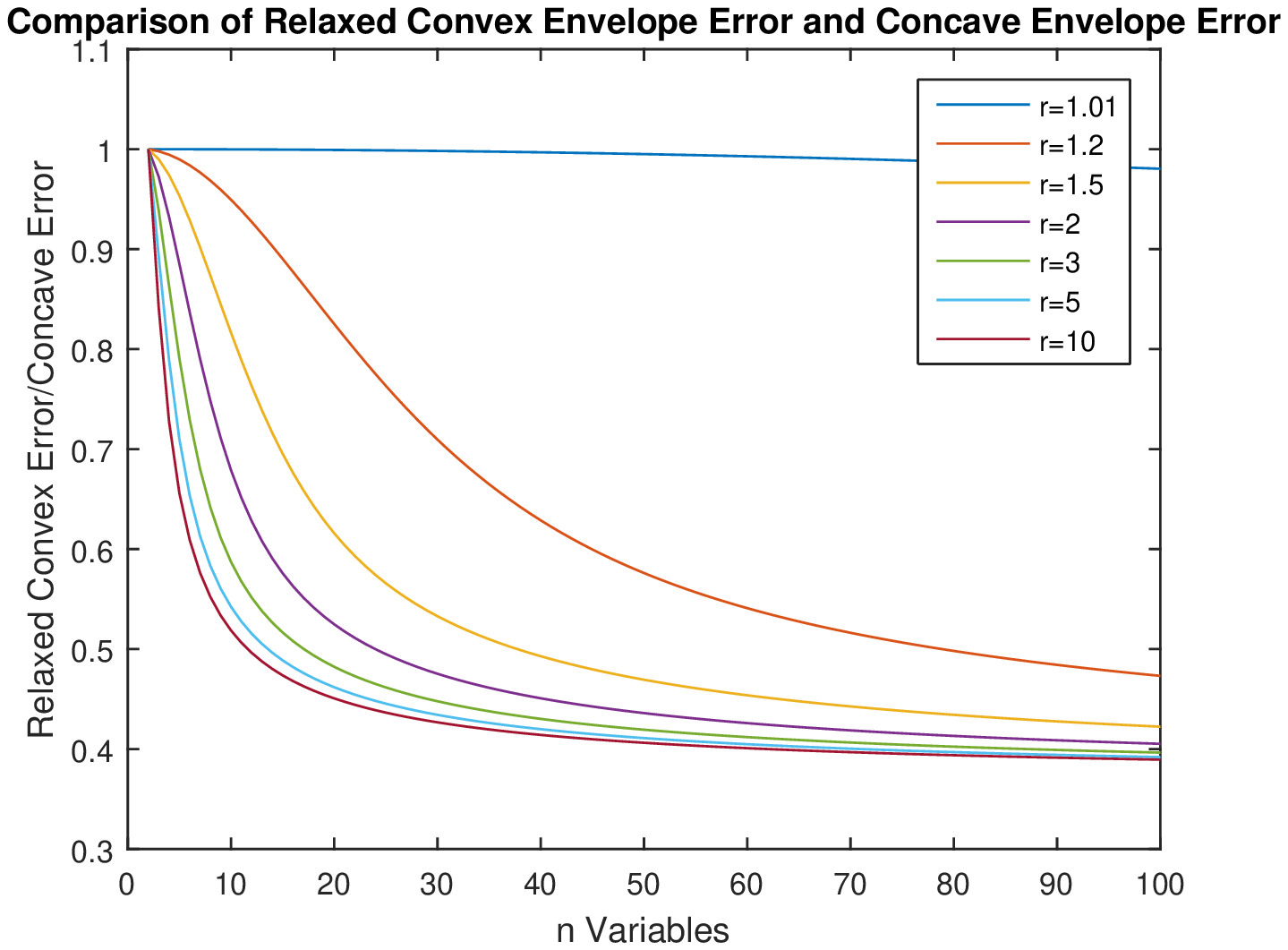}
\label{fig:RConvConc}
}
\caption{Error comparisons for $\multilin$ over $\Honer$.}
\label{fig:compare}
\end{figure}

Asymptotically, $\E$ dominates $\D$ in the following sense. \blue{Recall $t^{\ast}$ and $t^{\ast\ast}$  defined in Proposition~\ref{prop:D}}.

\begin{proposition}
$\lim_{n\to\infty} \frac{\E}{r^{n}-1} = 1$, and $\lim_{n\to\infty} \frac{\D}{r^{n}-1} \le \frac{1}{e}$ \blue{if $t^{\ast}\ge (n-1)/n$ or $t^{\ast\ast}\le (n-1)/n$}.
\end{proposition}
\begin{proof}
We have \[ \lim_{n\to\infty}\frac{\E}{r^{n}-1} \;=\; \frac{1}{r-1}\left[\lim_{n\to\infty} \frac{n-1}{n}\,\lim_{n\to\infty} {\left(\frac{r^n-1}{n(r-1)}\right)}^{\frac{1}{n-1}}\,-\,1\right] \;=\; \frac{1}{r-1}\left[1\cdot r - 1 \right] = 1.\] Proposition~\ref{prop:D} gives two bounds on $\D$. If $\D = (1 + \frac{n-1}{n}(r-1))^{n} - r^{n-1}$, then 
\[
\begin{split}
\lim_{n\to\infty}\frac{\D}{r^{n}-1} \;=\; \lim_{n\to\infty}\frac{\D}{r^{n}}\,\lim_{n\to\infty}\frac{r^{n}}{r^{n}-1} &\;=\; \lim_{n\to\infty}\left(\frac{1}{r} + \left(1-\frac{1}{n}\right)\left(1-\frac{1}{r}\right) \right)^{n} - \frac{1}{r} \\
&\;=\; \lim_{n\to\infty}\left(1 -\frac{1}{n}+\frac{1}{nr}  \right)^{n} - \frac{1}{r}\\
&\;=\;  e^{\frac{1}{r}-1}-\frac{1}{r} .
\end{split}\] The above function of $r$ is increasing over $[1,\infty)$ and converges to $1/e \approx 0.37$ as $r\to\infty$. The limit on the other value of $\D$ is $\frac{\ln r}{r-1}-\frac{1}{r}$ as $n\to\infty$, and the value of this function of $r$ never exceeds 0.22.
\end{proof}

Thus $\E$ seems to grow much more rapidly than $\D$ \blue{in some cases}.

%

\subsection{Symmetric box}\label{sec:err11}

\newcommand{\evenset}{\mathcal{N}^{even}}
\newcommand{\oddset}{\mathcal{N}^{odd}}

\subsubsection{Convex hull}
\citet{mahdi2010coloring} showed that the recursive McCormick relaxation, which \citet{ryoo2001analysis} had used to obtain an extended formulation of $\conv{\graph[\Hzone]}$, yields a compact extended formulation of $\conv{\graph[\Hnegoo]}$. However, to the best of our knowledge, there is no known characterization of this convex hull in the $(\bx,w)$-space. We provide this next. A different proof based on constructive arguments is presented in a companion paper \citep{symmetricpoly}.

\begin{theorem}\label{thm:convminusplus}
Partition subsets of $\{1,\dots,n\}$ into $\evenset := \{I\subseteq \{1,\dots,n\} \mid \abs{I}\text{ is even}  \}$ and $\oddset := \{I\subseteq \{1,\dots,n\} \mid \abs{I}\text{ is odd}  \}$. If $n$ is odd, then 
\[
\conv{\graph[\Hnegoo]} = \Big\{(x,w)\in \Hnegoo[n+1] \mid \; -(n-1) \,\le\, \sum_{i\in I} x_i \,-\, \sum_{i\notin  I} x_i \,+\, w \,\le\,  n-1, \ \,  I\in\evenset\Big\}.
\]
If $n$ is even, then 
\begin{equation*}
\begin{split}
\conv{\graph[\Hnegoo]} = \Big\{(x,w)\in\Hnegoo[n+1] \mid &\; \sum_{i\in I} x_i - \sum_{i\notin  I} x_i +  w\leq  n-1, \ \,  I\in\oddset\\
&\; \sum_{i\in I} x_i - \sum_{i\notin  I} x_i + w\geq -(n-1), \ \,  I \in\evenset\Big\}.
\end{split}
\end{equation*}
\end{theorem}
Before presenting our proof, we provide an intuition behind the proposed convex hull description. Denote $x_{n+1}=w$ to get $\graph[\Hnegoo] = \{x\in\Hnegoo[n+1]\mid x_{n+1}=\multilin \}$. It is well-known \citep{sherali1997convex,rikun1997convex} that for any box $\H$, the extreme points of $\conv{\graph[\Hnegoo]}$ are in bijection with the extreme points of $\H$ (this is also true for a multilinear polynomial). Hence the set of extreme points of $\conv{\graph[\Hnegoo]}$ is equal to $\graph[\Hnegoo] \cap \{-1,1\}^{n+1}$. A point in $\{-1,1\}^{n+1}$ violates $x_{n+1}=\multilin$ if and only if the set $\{i\in\{1,\dots,n+1\}\mid x_{i}=-1 \}$ has odd cardinality. Every such inadmissible point in $\{-1,1\}^{n+1}$ can be cut off using the ``no-good'' inequality \[
\sum_{i\in I}(x_{i}-(-1)) \:+\: \sum_{i\in\{1,\dots,n+1\}\setminus I}(1-x_{i}) \;\ge\; 2
\] 
for some odd subset $I\subseteq\{1,\dots,n+1\}$. The no-good cut for subset $I$ is valid to every point in $\{-1,1\}^{n+1}$, except that point which takes the value $-1$ at exactly those elements indexed by $I$. 
This cut rearranges to 
\begin{equation}\label{eq:nogood}
\sum_{i\in I}x_{i} \:-\: \sum_{\{1,\dots,n+1\}\setminus I}x_{i}\;\ge\; -(n-1).
\end{equation}
Hence $\conv{\graph[\Hnegoo]} = \conv{\{x\in\{-1,1\}^{n+1}\mid eq.~ \eqref{eq:nogood} \ \, \forall I\subseteq \{1,\dots,n+1\}, \text{ odd } \abs{I}   \}}$.  Consider the polytope
\begin{equation}\label{eq:convminusplus}
\P[]^{-1,1} := \{x\in\Hnegoo[n+1]\mid eq.~ \eqref{eq:nogood} \ \, \forall I\subseteq\{1,\dots,n+1\}, \text{ odd } \abs{I}  \},
\end{equation}
which is the LP relaxation of $\conv{\graph[\Hnegoo]}$. By construction, this polytope has the property that $\P[]^{-1,1}\cap\{-1,1\}^{n+1}\subseteq\graph[\Hnegoo]$. We will show in the proof of Theorem~\ref{thm:convminusplus} that the extreme points of $\P[]^{-1,1}$ are in $\{-1,1\}^{n+1}$, thereby implying that $\conv{\graph[\Hnegoo]}=\P[]^{-1,1}$. This equality, along with the following claim that is straightforward to verify, gives us the statement of Theorem~\ref{thm:convminusplus}.

\begin{observation}
After denoting $x_{n+1}=w$, each of the convex hull descriptions in Theorem~\ref{thm:convminusplus} becomes equal to the polytope $\P[]^{-1,1}$.  
\end{observation}


\begin{proof}[\textbf{Proof of Theorem~\ref{thm:convminusplus}}]
\renewcommand{\a}{c}
We show that for any $\a\in\real^{n+1}$, the linear program $z^{LP} = \max\{\a^{\top}x \colon x\in\P[]^{-1,1}\}$ has an optimal solution in $\graph[\Hnegoo]\cap\{-1,1\}^{n+1}$. We proceed by \blue{considering cases that are defined using} $A=\{i:\a_i=0\}, B=\{i:\a_i<0\}, C=\{i:\a_i>0\}$. Note two things: (1) $z^{LP} \le \sum_{i\in B\cup C}\abs{\a_{i}}$ due to $x\in\Hnegoo[n+1]$ for every feasible $x$, (2) any $x\in\{-1,1\}^{n+1}$ belongs to $\P[]^{-1,1}$ if and only if $\{i\in\{1,\dots,n+1\}\mid x_{i}=-1 \}$ has even cardinality.

\begin{description}
\item[Case 1: $\abs{B}$ is even.] 
Since $B$ has even cardinality, the point $x^{\ast}$ with $x^{\ast}_{i}=-1$ for $i\in B$ and $x^{\ast}_{i}=1$ for $i\in A\cup C$ belongs to $\P[]^{-1,1}$. This $x^{\ast}$ is optimal to $z^{LP}$ because $\a^{\top}x^{\ast} = \sum_{i\in B\cup C}\abs{\a_{i}}$.

\item[Case 2: $\abs{B}$ is odd and $\abs{A}\ge 1$.]
Choose an arbitrary $j_{0}\in A$ and set $x^{\ast}_{i}= -1$ for $i\in B\cup \{j_0\}$ and $x^{\ast}_{i}=1$ for $i\in (A\setminus\{j_{0}\})\cup C$ . This $x^{\ast}$  belongs to $\P[]^{-1,1}$ because $B\cup \{j_{0}\}$ is even and is optimal to $z^{LP}$ because $\a^{\top}x^{\ast} = \sum_{i\in B\cup C}\abs{\a_{i}}$.

\item[Case 3: $\abs{B}$ is odd and $\abs{A}=0$.]
Let $j_1\in \argmin_{1\leq i\leq n+1}\abs{\a_i}$. There are two subcases. When $j_1\in B$, i.e. $\a_{j_1}<0$, the point $x^\ast_{i}= -1$ for $i\in B\setminus \{j_1\}$ and $x^{\ast}_{i}=1$ for $i\in C\cup\{j_1\}$ is optimal with value $\sum_{i\in B\setminus \{j_1\}}\left(-\blue{\a}_i\right)+\sum_{i\in C\cup\{j_1\}}\blue{\a}_i$, because in this subcase 
\begin{align*}
\a^\top x &=\left(-\a_{j_1}\right)\left(\sum_{i\in C}x_i-\sum_{i\in B}x_i\right)+\sum_{i\in C}\left(\a_i+\a_{j_1}\right)x_i+\sum_{i\in B}\left(\a_i-\a_{j_1}\right)x_i\\
&\leq \left(-\a_{j_1}\right)(n-1)+\sum_{i\in C}\left(\a_i+\a_{j_1}\right)+\sum_{i\in B}\left(\a_{j_1}-\a_i\right)\\
&=\sum_{i\in C}\a_i+\sum_{i\in B}\left(-\a_i\right)+2\a_{j_1}\\
&=\sum_{i\in B\setminus \{j_1\}}\left(-\a_i\right)+\sum_{i\in C\cup\{j_1\}}\a_i,
\end{align*}
where the $\le$ inequality is obtained by applying \eqref{eq:nogood} with $I=B$. When $j_1\in C$, i.e. $\a_{j_1}>0$, then the point $x^\ast_{i}= -1$ for $i\in B\cup \{j_1\}$ and $x^{\ast}_{i}=1$ for $i\in C\setminus\{j_1\}$ is optimal with value $\sum_{i\in B\cup \{j_1\}}\left(-\a_i\right)+\sum_{i\in C\setminus\{j_1\}}\a_i$, because in this subcase 
\begin{align*}
\a^\top x&=\a_{j_1}\left(\sum_{i\in C}x_i-\sum_{i\in B}x_i\right)+\sum_{i\in C}\left(\a_i-\a_{j_1}\right)x_i+\sum_{i\in B}\left(\a_i+\a_{j_1}\right)x_i\\
&\leq \a_{j_1}(n-1)+\sum_{i\in C}\left(\a_i-\a_{j_1}\right)+\sum_{i\in B}\left(-\a_i-\a_{j_1}\right)\\
&=\sum_{i\in C}\a_i+\sum_{i\in B}\left(-\a_i\right)-2\a_{j_1}\\
&=\sum_{i\in B\cup \{j_1\}}\left(-\a_i\right)+\sum_{i\in C\setminus\{j_1\}}\a_i.
\end{align*}
\end{description}
This completes our proof for showing that $\P[]^{-1,1}$ has extreme points in $\{-1,1\}^{n+1}$.
\end{proof}

A scaling argument yields $\conv{\graph}$ when $\L_{i}=-\U_{i}$ for all $i$.

\newcommand{\y}{x^{\prime}}
\newcommand{\w}{w^{\prime}}
\subsubsection{Errors}
In order to prove Theorem~\ref{thm:converr11}, we make use of the reflection symmetry in the sets $\graph[\Hnegoo]$ and $\conv{\graph[\Hnegoo]}$, as described next. Let $\sign(\cdot)$ denote the sign of a scalar, with $\sign(0)$ considered positive. A point $(x,w)\in\Hnegoo[n+1]$ is said to have compatible signs if $\sign(w) = \sign(\multilin[x])$, i.e., $\sign(w)$ is negative if and only if $x$ has no zero entries and has an odd number of negative entries.  Define the following binary relation on $\Hnegoo[n+1]$: $(x,w) \sim (\y,\w)$ if (i) $|\w|=|w|$ and $|x_{j}| = |\y_{j}|$ for all $j$, and (ii) both $(x,w)$ and $(\y,\w)$ have compatible signs or both $(x,w)$ and $(\y,\w)$ do not have compatible signs.  Thus $(x,w)\sim(\y,\w)$ if and only if $\y$ is obtained from $x$ by reversing signs on odd (even) many entries of $x$ and setting $\w=-w$ ($\w=w$). This binary relation has two important properties. 
\begin{enumerate}
\item It preserves the error measure $h(x,w) := \abs{w - \multilin}$. Indeed, one can easily argue that $h(x,w) = h(\y,\w)$ if $(x,w)\sim(\y,\w)$. 
\item It is an equivalence relation, i.e., a reflexive symmetric transitive relation. This is obvious by construction of $\sim$.
\end{enumerate}
Now consider $[(x,w)] := \{(\y,\w)\in\Hnegoo\mid (x,w)\sim (\y,\w) \}$, the equivalence class of $(x,w)$ induced by $\sim$. Since $\sim$ is an equivalence relation on $\Hnegoo[n+1]$ and $\graph[\Hnegoo]$ and $\conv{\graph[\Hnegoo]}$ are subsets of $\Hnegoo[n+1]$, each of these sets is partitioned by $\sim$. Observe that the definition of $\sim$ means that for every $(\y,\w)\in\Hnegoo[n+1]$ having compatible (incompatible) signs, there exists $(x,w)\in\Hnegoo[n+1]$ such that $(\y,\w)\sim (x,w)$ and $(x,w)\ge\zerovec$ ($x\ge\zerovec, w < 0$). Now, because every point in $\graph[\Hnegoo]$ has compatible signs and $(x,w)\in\graph[\Hnegoo]$ trivially implies $[(x,w)]\subset\graph[\Hnegoo]$, we have 
\begin{subequations}
\begin{equation}
\graph[\Hnegoo] = \bigcup\,\left\{[(x,w)]\colon (x,w)\in\graph[\Hnegoo], (x,w)\ge\zerovec \right\}. 
\end{equation}
To make a similar statement for $\conv{\graph[\Hnegoo]}$, we need a small modification because the convex hull contains points with both compatible and incompatible signs. In particular, we must drop the nonnegativity requirement on $w$. Also, if $(x,w)\in\conv{\graph[\Hnegoo]}$, then using the fact that $(x,w)$ is a convex combination of points in $\graph[\Hnegoo]$, all of which have compatible signs, we get that $[(x,w)]\subset\conv{\graph[\Hnegoo]}$. Thus we have the following:
\begin{equation}\label{eq:sym11a}
\conv{\graph[\Hnegoo]} = \bigcup\,\left\{[(x,w)]\colon (x,w)\in\conv{\graph[\Hnegoo]}, x\ge\zerovec \right\}
\end{equation}
Now, the fact that $\sim$ is error-preserving leads to
\begin{equation}\label{eq:sym11b}
\err{\conv{\graph[\Hnegoo]}} = \max\left\{\,\abs{w-\multilin} \colon (x,w)\in \conv{\graph[\Hnegoo]}, x \ge \zerovec\right\},
\end{equation}
meaning that we only need to consider nonnegative values of $x$ when computing the convex hull error.
\end{subequations}

\begin{proof}[\textbf{Proof of Theorem~\ref{thm:converr11}}]
To upper bound the convex hull error. We only present arguments for when $n$ is odd, since the even case is almost exactly the same due to similar characterizations of the convex hulls in Theorem~\ref{thm:convminusplus}. By equation~\eqref{eq:sym11b}, we consider only $(x,w)\in\conv{\graph[\Hnegoo]}$ with $x\ge\zerovec$. Thus, $\err{\conv{\graph[\Hnegoo]}}$ is equal to the maximum of the maximum errors of $\cvxenv{\f[m]}{\Hnegoo}(x)$ and $\concenv{\f[m]}{\Hnegoo}(x)$ calculated over $\Hzone$.
\[ 
\begin{split}
\multilin \,-\, \cvxenv{\f[m]}{\Hnegoo}(x) &\;=\; \multilin \,-\, \max\left\{-(n-1)\,+\,\max_{I\in\evenset}\sum_{j\notin I}x_{j} - \sum_{j\in I}x_{j} ,\, -1 \right\} \\
&\;=\; \multilin \,-\, \max\left\{\sum_{j=1}^{n}x_{j} - (n-1),\, -1 \right\}\\
&\;=\; \min\left\{\multilin - \sum_{j=1}^{n}x_{j} + (n-1)  ,\, \multilin + 1 \right\} \\
&\;\le\; \min\left\{\multilin - n\sqrt[\leftroot{-1}\uproot{10}n]{\multilin} + (n-1)  ,\, \multilin + 1 \right\}, \\
\end{split}
\] where $x\ge\zerovec$ has given us the second equality, and the inequality in the last step  from applying the arithmetic-geometric means inequality. Therefore, after regarding $\sqrt[n]{\multilin}$ as a scalar variable $t$, we get $\max_{t\in[0,1]}\min\varphi(t)$ to be an upper bound on the convex envelope error, where $\varphi(t)= \min\{t^{n} - nt + n-1 , t^{n}+1\}$. The function $t^{n} - nt + n-1$ is convex \blue{decreasing} on $[0,1]$ whereas $t^{n}+1$ is convex \blue{increasing} on $[0,1]$, and hence the maximum value of $\varphi$ on $[0,1]$ occurs at a breakpoint where the two functions have equal value. Solving for $t^{n} - nt + n-1 = t^{n}+1$ yields $t = 1 - 2/n$, and so the upper bound is $1 + (1 - 2/n)^{n}$. This bound is tight since it is attained at $x=(1-2/n)\onevec$ where $\cvxenv{\f[m]}{\Hnegoo}((1-2/n)\onevec)=-1$. On the concave side, we have $\concenv{\f[m]}{\Hnegoo}(x) \le 1$ and since $\multilin\ge 0$ for $x\ge\zerovec$, the concave envelope error is upper bounded by $1$. Thus, $\err{\conv{\graph[\Hnegoo]}} = \max\{1 + (1 - 2/n)^{n} , 1\} = 1 + (1-2/n)^{n}$.

To find the points where this bound is attained, we already observed the point $(x,w) = ((1-2/n)\onevec,-1)$. Since our relation $\sim$ is error-preserving, all points in the equivalence class of $((1-2/n)\onevec,-1)$ have the same error, and there are $2^{n}$ many such points. Finally, note that for any point $(\y,1)\in[((1-2/n)\onevec,-1)]$, the above bounds on the envelopes would be reversed so that both the envelopes have the same maximum error over the entire $\Hnegoo$ box.  
\end{proof}

%

\section*{Acknowledgements}
The first author was supported in part by ONR grant N00014-16-1-2168. The second author was supported in part by ONR grant N00014-16-1-2725. \blue{We thank two referees whose meticulous reading helped us clarify some of the technical details.}

\begin{appendices}

\section{Missing Proofs}\label{sec:app1}
\begin{proof}[\textbf{Proof of Proposition~\ref{prop:sigmad}}]
Since $\S\subseteq\Hzone$ and $\beta\ge\onevec$ make $\monom\le\min_{j}x_{j}<\beta^{\top}\bx$ for all $\bx\in\S$, we have $\sigma(\beta) < \degree[\beta]$. The lower bound of 0 comes from 
\[
\sigma(\beta) \;\ge\; \degree[\beta] + \min_{\bx\in\S}\monom - \max_{\bx\in\S} \beta^{\top}\bx \;\ge\; \degree[\beta] + 0 - \degree[\beta] = 0. 
\] 
%
If $\S=\Hzone$, then $\onevec\in\S$ implies that $\min_{\bx\in\S}\monom - \beta^{\top}\bx \le 1 - \degree[\beta]$ and so by \eqref{eq:sigma}, we have $\sigma(\beta)\le 1$. For the fourth claim we have $\Delta^{\zerovec}_{n}\cap\Delta^{\onevec}_{n}=\{x\in\Hzone\mid \sum_{i}x_{i}=n-1\} = \conv{\{\onevec-\onevec[1],\ldots,\onevec-\onevec[n]\}}$. Denote this simplex by $\Delta^{\onevec}_{n-1}$. The assumption $\Delta^{\onevec}_{n-1}\subseteq\S$ means that $\onevec-\onevec[i]\in\S$ for all $i$. Substituting this point into \eqref{eq:sigma} gives us $\sigma(\beta) \le \sum_{j=1}^{n}\beta_{j} + 0 - \sum_{j\neq i}\beta_{j} = \beta_{i}$ for all $i$. This leads to $\sigma(\beta)\le\min_{j}\beta_{j}$. Since  $\Delta^{\onevec}_{n-1}\subseteq\S\subseteq\Delta^{\zerovec}_{n}$, $\max_{x \in \Delta^{\onevec}_{n-1}}\beta^{\top} x \le\max_{x\in\S} \beta^{\top} x \le \max_{x \in \Delta^{\zerovec}_{n}}\beta^{\top} x$. Note that $\Delta^{\zerovec}_{n} = \conv{(\{x\in\{0,1\}^{n}\mid \sum_{i}x_{i}\le n-2 \}\cup \Delta^{\onevec}_{n-1})}$. The positivity of $\beta$ then makes it clear that $\max_{x\in\Delta^{\zerovec}_{n}}\beta^{\top} x = \max_{x \in \Delta^{\onevec}_{n-1}}\beta^{\top} x$. Hence $\max_{x\in\S} \beta^{\top} x = \max_{x \in \Delta^{\onevec}_{n-1}}\beta^{\top} x = \sum_{j=1}^{n-1}\beta_{(j)}$, where $\beta_{(1)} \ge \beta_{(2)} \ge \cdots \ge \beta_{(n)}$. Now,
\[
\sigma(\beta) \;\ge\; \sum_{j=1}^{n}\beta_{j} \;+\; \min_{x\in\S}\,\monom \; - \; \max_{x\in\S} \beta^{\top} x \;\ge\; \sum_{j=1}^{n}\beta_{j} \;+\; 0 \;-\; \sum_{j=1}^{n-1}\beta_{(j)}
\;=\; \beta_{(n)} \;=\; \min_{j}\beta_{j}.
\]
Since we have already argued $\sigma(\beta)\le\min_{j}\beta_{j}$, it follows that $\sigma(\beta)=\min_{j}\beta_{j}$.
\end{proof}

\begin{proof}[\textbf{Proof of Lemma~\ref{lem:exp1}}]
\renewcommand{\chi}{\sigma}
\renewcommand{\beta}{\lambda}
\renewcommand{\xi}{t}
For nontriviality, assume $\beta_{1}>1$. 

(1) The first derivative is $\phi^{\prime}(\chi) = -\beta_{1}(1-\chi)^{\beta_{1}-1} + \beta_{2}$. If $\beta_{2}\ge\beta_{1}$, then $\phi^{\prime}(0) \ge 0$ and $\phi^{\prime}(\chi) > 0$ for all $\chi \in (0, 1]$ and hence $\phi$ is strictly increasing over $(0,1)$ and $\phi(\chi) > \phi(0) = 0$ for all $\chi\in(0,1]$. 

(2 \& 3) Now assume $1\le\beta_{2} < \beta_{1}$.  Set $\tilde{\chi} = 1-(\beta_{2}/\beta_{1})^{\frac{1}{\beta_{1}-1}}$ and realize that $\phi^{\prime}(\tilde{\chi})=0$ and $\tilde{\chi}\in(0,1)$. Then we have $\phi^{\prime}(\chi) < 0$ for $\chi \in (0, \tilde{\chi})$. Therefore $\phi$ is decreasing on $(0, \tilde{\chi}]$, which implies $\phi(\chi)< \phi(0)=0$ for $\chi \in (0, \tilde{\chi}]$. Hence $\phi(\tilde{\chi})<0$.  The construction of $\tilde{\chi}$ also implies $\phi^{\prime}(\chi) > 0$, and hence $\phi$ is increasing, for $\chi \in ( \tilde{\chi},1]$. Since $\phi(1)\ge0$, it follows that there is a unique real number $\chi^{\ast}$ in $(\tilde{\chi},1]$ such that $\phi(\chi^{\ast})=0$. Thus we have $\phi(\chi)\le 0$ for $\chi\in[0,\chi^{\ast}]$ and $\phi(\chi)>0$ for $\chi\in(\chi^{\ast},1]$. If $\beta_{1}$ is odd, the other root is obtained by applying Descartes' rule of signs as in the first claim.

(4) Take $\beta\in(\beta_{2},\infty)$ and define $g(\chi) := (1-\chi)^{\beta_{1}} + \beta\chi - 1$. If $\beta\ge\beta_{1}$, then the first claim in this lemma, with $\beta_{2}$ replaced by $\beta$, gives us $g(\chi)\ge0$. Now assume $\beta<\beta_{1}$. Applying the second claim in this lemma, after replacing $\beta_{2}$ with $\beta$, tells us there is a unique real $\chi^{\ast\ast}$ that is a root of $g$ in $(0,1]$. Now $g(\chi^{\ast})=\phi(\chi^{\ast})+(\beta-\beta_{2})\chi^{\ast}>0$ because $\phi(\chi^{\ast})=0,\beta>\beta_{2},\chi^{\ast}>0$. Then the third claim in this lemma, with $\beta_{2}$ replaced by $\beta$, gives us $\chi^{\ast}>\chi^{\ast\ast}$ and consequently, the proposed fourth claim.

For the final part, note that the roots of $\phi$ and its complemented polynomial $\phi^{\prime}(\xi) := \xi^{\beta_{1}} -\beta_{2}\xi+\beta_{2}-1$ are in bijection under the relation $\sigma=1-\xi$. Descartes' rule of signs tells us that $\phi^{\prime}$ has exactly one positive root besides $\xi=1$. When $\beta_{2}>\beta_{1}$, this root must be in $(1,\infty)$ because otherwise we would get a contradiction to $\phi$ not having any roots in $(0,1]$. Descartes' rule also tells us there is exactly one negative root when $\beta_{1}$ is odd. This translates to $\phi$ having a root in $(1,\infty)$ if and only if $\beta_{1}$ is odd.
\end{proof}

\begin{proof}[\textbf{Proof of Proposition~\ref{prop:D}}]
Note that $\psi(0)=\psi(1)=0$. We first claim that $\psi$ is strictly increasing on $(0,t^{\ast})$. In fact, we argue the stronger claim that $\psi(t) > 0$ for all $t\in(0,1)$. This claim is equivalent to showing that $(\frac{1 + (r-1)t}{r^{t}})^{n} > 1$, which is equivalent to $r^{t} - (r-1)t - 1 < 0$. The function $t \mapsto r^{t} - (r-1)t - 1$ is convex and is zero-valued at $t=0$ and $t=1$. Therefore, by convexity, $r^{t} - (r-1)t - 1 < 0$ for all $t\in(0,1)$, and hence, we have $\psi(t) > 0$ for all $t\in(0,1)$.

Since $\D \le \max_{t\in[0,1]} \psi(t)$, $\psi$ is strictly increasing on $(0,t^{\ast})$, \blue{and $\psi(1)=0$}, the condition $t^{\ast}\ge (n-1)/n$ implies that $i=(n-1)$ yields the maximum value in the formula for $\D$. Now suppose $\blue{t^{\ast\ast}} \le (n-1)/n$. Since $t^{\ast\ast}$ is a stationary point, $(1+\blue{t^{\ast\ast}}(r-1))^{n-1}=r^{n\blue{t^{\ast\ast}}}\frac{\ln r}{r-1}$. Now,
\[
\begin{split}
0 \;<\; \D \;\le\; (1+\blue{t^{\ast\ast}}(r-1))^n-r^{n\blue{t^{\ast\ast}}} \;=\; r^{\frac{n^2\blue{t^{\ast\ast}}}{n-1}}\left(\frac{\ln r}{r-1}\right)^{\frac{n}{n-1}} -r^{n\blue{t^{\ast\ast}}} &\;=\; r^{n\blue{t^{\ast\ast}}}\left(r^{\frac{n\blue{t^{\ast\ast}}}{n-1}}\left(\frac{\ln r}{r-1}\right)^{\frac{n}{n-1}}-1\right) \\
&\le r^{n-1}\left(r \left(\frac{\ln r}{r-1}\right)^{\frac{n}{n-1}}-1\right),
\end{split}
\] where the last inequality uses $n\blue{t^{\ast\ast}} \le n-1$ and $r>1$. \blue{Finally, if $t^{\ast} < (n-1)/n < t^{\ast\ast}$, since $t^{\ast\ast}$ can be arbitrarily close to $1$, we can only bound $r^{nt^{\ast\ast}}$ and $r^{\frac{nt^{\ast\ast}}{n-1}}$ in above by $r^{n}$ and $r^{\frac{n}{n-1}}$, respectively, to obtain the last proposed bound on $\D$.}
\end{proof}

\end{appendices}

\printbibliography

\end{document}